\documentclass{amsart}

\usepackage{amsthm, amssymb, amsmath}
\usepackage{a4wide}
\usepackage{array}
\usepackage[dvipsnames]{xcolor}
\usepackage[english]{babel}
\usepackage{graphicx, tikz, pgfplots}

\usepackage{subfigure}
\usepackage{url}
 \usepackage[foot]{amsaddr}
\usepackage{fancyhdr}
\usetikzlibrary{fit}

\definecolor{apricot}{rgb}{0.98, 0.81, 0.69}
\definecolor{aquamarine}{rgb}{0.5, 1.0, 0.83}
\definecolor{babyblueeyes}{rgb}{0.63, 0.79, 0.95}
\definecolor{bananamania}{rgb}{0.98, 0.91, 0.71}
\definecolor{bittersweet}{rgb}{1.0, 0.44, 0.37}
\definecolor{bluebell}{rgb}{0.64, 0.64, 0.82}
\definecolor{blush}{rgb}{0.87, 0.36, 0.51}

\definecolor{cerulean}{rgb}{0.0, 0.48, 0.65}
\definecolor{darkcyan}{rgb}{0.0, 0.55, 0.55}

\definecolor{airforceblue}{rgb}{0.36, 0.54, 0.66}
\definecolor{antiquefuchsia}{rgb}{0.57, 0.36, 0.51}
\definecolor{asparagus}{rgb}{0.53, 0.66, 0.42}
\definecolor{ballblue}{rgb}{0.13, 0.67, 0.8}
\definecolor{blueviolet}{rgb}{0.54, 0.17, 0.89}
\definecolor{brightgreen}{rgb}{0.4, 1.0, 0.0}
\definecolor{brightcerulean}{rgb}{0.11, 0.67, 0.84}
\definecolor{brightpink}{rgb}{1.0, 0.0, 0.5}
\definecolor{brightturquoise}{rgb}{0.03, 0.91, 0.87}
\definecolor{brightube}{rgb}{0.82, 0.62, 0.91}
\definecolor{brilliantlavender}{rgb}{0.98, 0.86, 1.0}
\definecolor{bluebell}{rgb}{0.64, 0.64, 0.82}
\definecolor{bluegray}{rgb}{0.4, 0.6, 0.8}
\definecolor{bluegreen}{rgb}{0.0, 0.87, 0.87}
\definecolor{brandeisblue}{rgb}{0.0, 0.44, 1.0}
\definecolor{capri}{rgb}{0.0, 0.75, 1.0}
\definecolor{fandango}{rgb}{0.71, 0.2, 0.54}

\definecolor{bubblegum}{rgb}{0.99, 0.76, 0.8}
\definecolor{brilliantrose}{rgb}{1.0, 0.33, 0.64}
\definecolor{brightmaroon}{rgb}{0.76, 0.13, 0.28}

\newtheorem{theorem}{Theorem}[section]

\newtheorem{corollary}[theorem]{Corollary}
\newtheorem{definition}{Definition}

\newtheorem{conjecture}{Conjecture}

\begin{document}
\title{Alternating $N$-expansions}
\author{Karma Dajani and Niels Langeveld}
\address[Karma Dajani]{Department of Mathematics, Utrecht University, P.O.~Box 80010, 3508TA Utrecht, the Netherlands}
\email[Karma Dajani]{k.dajani1@uu.nl}
\address[Niels Langeveld]{ Montanuniversit\"at Leoben, Department Mathematik und Informationstechnologie,
Franz-Josef-Strasse 18
 A-8700 Leoben
AUSTRIA}
\email[Niels Langeveld]{niels.langeveld@unileoben.ac.at}

\date{Version of \today}

\subjclass[2010]{37E05, 28D05, 37E15, 37A45, 37A05}
\keywords{$N$-continued fractions, invariant density, ergodicity, natural extensions}

\maketitle
 
\begin{abstract}
We introduce a family of maps generating continued fractions where the digit $1$ in the numerator is replaced cyclically by some given non-negative integers $(N_1,\ldots,N_m)$.  We prove the convergence of the given algorithm, and study the underlying dynamical system generating such expansions. We prove the existence of  a unique absolutely continuous invariant ergodic measure. In special cases, we are able to build the natural extension and give an explicit expression of the invariant measure. For these cases, we formulate a Doeblin-Lenstra type theorem. For other cases we have a more implicit expression that we conjecture gives the invariant density. This conjecture is supported by simulations. For the simulations we use a method that gives us a smooth approximation in every iteration. 
\end{abstract}

\maketitle
\section*{Introduction}

In this article, we study a variation of the $N$-continued fraction expansions introduced in 2008 by Burger et al.~\cite{BGRKWY208}. These expansions have the form
\begin{equation}
x=\cfrac{N}{d_1+\cfrac{N}{d_2+\cfrac{N}{\ddots}}},\label{introductionexpansion}
\end{equation}
where the digits $d_i\in\mathbb{N}$, for $i\geq 1$.
Seemingly an innocent change, the effects are dramatic: no point has a unique expansion (except 0), quadratic irrationals seem to have all sorts of periodic and non-periodic as well as chaotic expansions, and the associated dynamics are much more complicated.  These expansions are generated iteratively by applying, in a deterministic or random manner, at each step of the expansion one of the maps  $S_k(x)=\dfrac{N}{x}-k$, for $k\ge 1$. 
The periodicity of the $N$-expansions of quadratic irrationals was partially studied in~\cite{BGRKWY208, AW11}. The ergodic properties of some deterministic $N$-expansion algorithms, defined as iterations of an appropriate transformation, were investigated in~\cite{DKW13}, and invariant measures of the random algorithm generating all $N$-expansions were given in~\cite{DO18}. In~\cite{KL17}, the authors concentrated on expansions of points in windows of length 1 to obtain a parametrised family of $N$-expansions with finitely many digits. For certain values of the parameter they were able to obtain the density of the invariant measure. For other values of the parameter where the invariant density could not be found, they gave some numerical estimation of the entropy.

\medskip
Inspired by alternating $\beta$-expansions (see~\cite{CCD2021}) we study alternating $N$-expansions.  These are generated by applying periodically one of a given finite set of transformations defined on distinct intervals, and each transformation has a range equal to the domain of the following transformation. This procedure leads to continued fractions whose numerators  are periodic and the digits are periodically taken from different digit sets.
To be more precise, for $i\in \{1,\ldots, m\}$ let $I_i:=[a_i,a_i+1)$ with $a_i\in\mathbb{N}\cup \{0\}$  be a set of distinct intervals. We will study continued fraction maps $T:\cup I_i \rightarrow \cup I_i$ such that $T(I_i)=I_{(i \mod m) +1}$ for $i\in \{1,\ldots,m\}$. On every interval $I_i$ we choose an $N\in\mathbb{N}$. Let $j(x)=i$ when $x\in I_i$  and  $k(x)= (j(x) \mod m) +1$. Furthermore,  let $\textbf{N}=(N_1,\ldots, N_{m})\in \mathbb{N}^m$. We define  $T:\cup I_i \rightarrow \cup I_i$ as 

\begin{equation}\label{definitionT}
T(x)=\frac{N_{j(x)}}{x}- \left\lfloor \frac{N_{j(x)}}{x} \right\rfloor +a_{k(x)},
\end{equation}
for $x\neq 0$ and $T(0)=0$ whenever $0\in \Omega$. Note that not for every  $\textbf{N}\in \mathbb{N}^m$ this gives us a continued fraction algorithm with only positive digits. This leads us to the following definition:
\begin{definition}[allowable]
Let $I_i:=[a_i,a_i+1)$ with $a_i\in\mathbb{N}\cup \{0\}$  be a set of distinct intervals and $\Omega=\cup_{i=1}^{m} I_i$. We call  $\textbf{N}=(N_1,\ldots, N_{m})\in \mathbb{N}^m$ \textbf{allowable} for $\Omega$ if for all $x\in \Omega$ we have $ \left\lfloor \frac{N_{j(x)}}{x} \right\rfloor -a_{k(x)}>0$.
\end{definition}
Of course, since we can pick the values for $\textbf{N}$ arbitrarily high, we can find a continued fraction algorithm for every choice of distinct intervals with all possible ways of indexing them. Many of these choices will give rise to a dynamical system that is difficult to study since not all branches will be full (a branch of the map $T$ will not be mapped entirely onto the next interval). After proving convergence of associated continued fractions for these general systems, as well as the existence of a unique absolutely continuous invariant ergodic measure, we will shift our focus on choices of $\textbf{N}$ such that we do have full branches. This is desirable for many reasons. Furthermore, we also study the more simple situation for which all values in $\textbf{N}$ are the same. In this case the numerators will not alternate but the sets from which the digits are taken are alternating.
\begin{definition}[desirable and simple]
Let
 $I_i:=[a_i,a_i+1)$ with $a_i\in\mathbb{N}\cup \{0\}$  be a set of distinct intervals and $\Omega=\cup_{i=1}^{m} I_i$ . We call  $\textbf{N}=(N_1,\ldots, N_{m})\in \mathbb{N}^m$ \textbf{desirable} for $\Omega$ if it is allowable and for all $i\in\{1,\ldots,m\}$ we have that $a_i\mid N_i$ and  $a_i+1 \mid N_i$ whenever $a_i\neq0$. Furthermore, we call  $\textbf{N}=(N_1,\ldots, N_{m})\in \mathbb{N}^m$  \textbf{simple} if it is desirable and $N_i=N_j$ for all $i,j\in\{1,\ldots,m\}$.
\end{definition}
Note that for any set of distinct intervals we can find a simple $\textbf{N}$ by taking a large enough multiple of the product of the endpoints of the intervals $I_i$, omitting $0$.  With the aid of some examples we will go through the different scenarios as well as introducing notation and explaining why these systems give rise to continued fractions (this will be done in Section~\ref{sec:examples}). In Section~\ref{sec:convetc} we will prove that the convergents of these continued fractions indeed converge. In Section~\ref{sec:ergodic} we prove that our underlying system has a unique invariant absolutely continuous ergodic measure. In Section~\ref{sec:naturalextension} we give a planar version of the natural extension for the simple (two interval) case and from this construction we are able to exhibit an explicit expression for the invariant density. Furthermore, the natural extension allows us to prove a Doeblin-Lenstra type theorem which tells us something about the quality of the convergents for generic points. In the case of more intervals we are unable to determine a suitable planar domain for the natural extension. We found a candidate for the domain but we were not able to prove that this domain has a positive measure. With simulations we support our conjecture that the candidate we found is indeed correct. For the simulations we provide an algorithm that approximates the domain from above. This algorithm is fast and each iteration will give rise to a smooth approximation of the invariant measure.

\medskip
We could have chosen not to take $a_i$ as distinct integers but as real number such that the corresponding intervals are distinct. Though, we choose to focus our attention to the case that $a_i$ are distinct integers since this already yields interesting results. One can also study $N$-expansions with $N$ being a positive real number (greater than one). This is done for example in~\cite{GS17,M20}. These continued fractions we also leave out of the scope of this article.



\bigskip
\section{Examples}\label{sec:examples}

\textbf{Example 1:} Let $a_1=1$, $a_2=3$ and $\textbf{N}=(9,12)$. We find that $ T: [1,2)\cup[3,4)\rightarrow [1,2)\cup[3,4)$ is given by
\[
 T(x)=\begin{cases}
\frac{9}{x}- \lfloor \frac{9}{x}\rfloor+3  & \text{ for } x\in [1,2),\\
\frac{12}{x}- \lfloor \frac{12}{x}\rfloor+1  & \text{ for } x\in [3,4),\\
\end{cases}
\]
 see Figure \ref{fig:Texample1}.

\begin{figure}[h]
\centering
\begin{tikzpicture}[scale=1.33333]
\draw(1,1)node[below]{\small $1$}--(4,1)node[below]{\small $4$}--(4,4)--(1,4)node[left]{\small $4$}--(1,1);
 \draw[dotted] (9/5,1)node[below]{\small $\frac{9}{5}$}--(9/5,4);
 \draw[dotted] (9/6,1)node[below]{\small $\frac{9}{6}$}--(9/6,4);
  \draw[dotted] (9/7,1)node[below]{\small $\frac{9}{7}$}--(9/7,4);
   \draw[dotted] (9/8,1)node[below]{\small $\frac{9}{8}$}--(9/8,4);
 \draw[dotted] (2,1)node[below]{\small $2$}--(2,4);
 \draw[dotted] (3,1)node[below]{\small $3$}--(3,4);
  \draw[dotted] (1,2)node[left]{\small $2$}--(4,2);
 \draw[dotted] (1,3)node[left]{\small $3$}--(4,3);


\draw[thick, purple!50!black, smooth, samples =20, domain=3:4] plot(\x,{12 / \x -2)});
\draw[thick, purple!50!black, smooth, samples =20, domain=1.8:2] plot(\x,{9 / \x -1)});
\draw[thick, purple!50!black, smooth, samples =20, domain=9/6:1.8] plot(\x,{9 / \x -2)});
\draw[thick, purple!50!black, smooth, samples =20, domain=9/7:9/6] plot(\x,{9 / \x -3)});
\draw[thick, purple!50!black, smooth, samples =20, domain=9/8:9/7] plot(\x,{9 / \x -4)});
\draw[thick, purple!50!black, smooth, samples =20, domain=1:9/8] plot(\x,{9 / \x -5)});
\end{tikzpicture}
\caption{The map $T$ of Example 1.}
\label{fig:Texample1}
\end{figure}
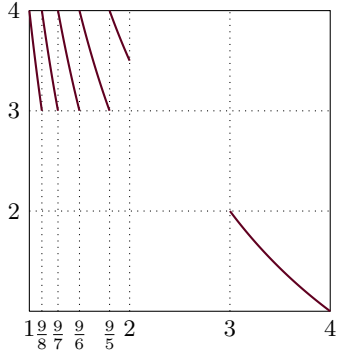

 We see that, since $3\mid12$ and $4\mid12$ we have full branches on $[3,4)$ but since $2\nmid 9$ we do not have full branches on $[1,2)$. Therefore, this choice of $\textbf{N}$ is allowable but not desirable. Note that at the left end point of any of the intervals, $T$ is not continuous giving us an extra digit in case the left point is not equal to $0$.  Setting $d_1(x)=\lfloor \frac{N_{j(x)}}{x} \rfloor -a_{k(x)}$ and $d_n(x)=d_1(T^{n-1}(x))$  we have
\[
T(x)=\frac{N_{j(x)}}{x}- d_1(x)
\]
and so
\[
x=\frac{N_{j(x)}}{d_1(x)+T(x)}.
\]
We can replace $T(x)$ by $T(x)=\frac{N_{j(T(x))}}{d_2(x)+T^2(x)}$ and continue in this manner to find
\[
x=\frac{N_{j(x)}}{\displaystyle d_1(x)+\frac{N_{j(T(x))}}{\displaystyle  d_2(x)+\frac{N_{j(T^2(x))}}{\ddots}}}.
\]
For $x\in [1,2)$ we find
 \[
x=\frac{9}{\displaystyle d_1(x)+\frac{12}{\displaystyle  d_2(x)+\frac{9}{\ddots}}},
\]
with  $d_i \in \{1,\ldots,6\} $ for $i$ odd (only $6$ when $T^{i-1}(x)=1$) and $d_i \in \{2,3\} $ for $i$ even (only $3$ when $T^{i-1}(x)=3$). For $x\in[3,4)$ we find
\[
x=\frac{12}{\displaystyle d_1(x)+\frac{9}{\displaystyle  d_2(x)+\frac{12}{\ddots}}},
\]
with $d_i\in \{2,3\}$ for $i$ odd and $d_i \in \{1,\ldots,6\}$ for $i$ even.\newline

\medskip
\textbf{Example 2}: Let $a_1=1, a_2=2$ and $\textbf{N}=(8,12)$. Then we can define $T:[1,3)\rightarrow[1,3)$ as  
\[
 T(x)=\begin{cases}
\frac{8}{x}- \lfloor \frac{8}{x}\rfloor+2 &  \text{ for } x\in [1,2),\\
\frac{12}{x}- \lfloor \frac{12}{x}\rfloor+1  & \text{ for } x\in [2,3),\\
\end{cases}
\]
see Figure \ref{fig:Texample2}. Note that this choice of $\textbf{N}$ is desirable.

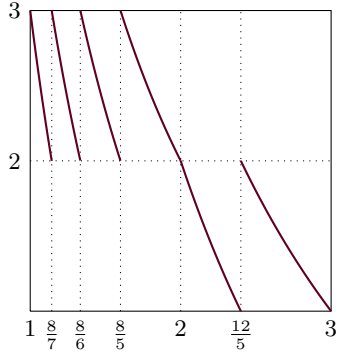
\begin{figure}[h]
\centering
\begin{tikzpicture}
\draw[white] (-0.5,0)--(3,0);
\draw(0,0)node[below]{\small $1$}--(4,0)node[below]{\small $3$}--(4,4)--(0,4)node[left]{\small $3$}--(0,0);
 \draw[dotted] (2.8,0)node[below]{\small $\frac{12}{5}$}--(2.8,4);
  \draw[dotted] (1.2,0)node[below]{\small $\frac{8}{5}$}--(1.2,4);
  \draw[dotted] (.666,0)node[below]{\small $\frac{8}{6}$}--(.666,4);
 \draw[dotted] (0.286,0)node[below]{\small $\frac{8}{7}$}--(0.286,4);
 \draw[dotted] (2,0)node[below]{\small $2$}--(2,4);
 \draw[dotted] (0,2)node[left]{\small $2$}--(4,2);

 \draw[thick, purple!50!black, smooth, samples =20, domain=2.8:4] plot(\x,{2*(12 / (0.5*\x+1) -4)});
\draw[thick, purple!50!black, smooth, samples =20, domain=2:2.8] plot(\x,{2*(12 / (0.5*\x+1) -5)});

\draw[thick, purple!50!black, smooth, samples =20, domain=1.2:2] plot(\x,{2*(8 / (0.5*\x+1) -3)});
\draw[thick, purple!50!black, smooth, samples =20, domain=.666:1.2] plot(\x,{2*(8 / (0.5*\x+1) -4)});
\draw[thick, purple!50!black, smooth, samples =20, domain=.286:.666] plot(\x,{2*(8 / (0.5*\x+1) -5)});
\draw[thick, purple!50!black, smooth, samples =20, domain=0:.286] plot(\x,{2*(8 / (0.5*\x+1) -6)});

\end{tikzpicture}
\caption{The map $T$ of Example 2.}
\label{fig:Texample2}
\end{figure}

For $x\in[1,2)$ we find the continued fraction
\[
x=\frac{8}{\displaystyle d_1+\frac{12}{\displaystyle  d_2+\frac{8}{\ddots}}},
\]
where $d_i\in \{2,3,4,5,6\}$ for $i$ odd (only $6$ when $T^{i-1}(x)=1$) and $d_i\in\{3,4,5\}$ for $i$ even (only $5$ when $T^{i-1}(x)=2$).\newline

\medskip
\textbf{Example 3}:  Let $a_1=0$, $a_2=2$, $a_3=1$, $a_4=3$. This time we choose $\textbf{N}=(12,12,12,12)$ so that we are in the simple case. Note that since there is an $i$ such that $a_i=0$ we have infinitely many branches. We find that $T:[0,4)\rightarrow[0,4)$ is now defined as 
\[
 T(x)=\begin{cases}
\frac{12}{x}- \lfloor \frac{12}{x}\rfloor+2 & \text{ for } x\in [0,1),\\
\frac{12}{x}- \lfloor \frac{12}{x}\rfloor+3 &  \text{ for } x\in [1,2),\\
\frac{12}{x}- \lfloor \frac{12}{x}\rfloor+1 & \text{ for } x\in [2,3),\\
\frac{12}{x}- \lfloor \frac{12}{x}\rfloor &  \text{ for } x\in [3,4),\\
\end{cases}
\]
see Figure \ref{fig:Texample3}. Note that this system is simple.

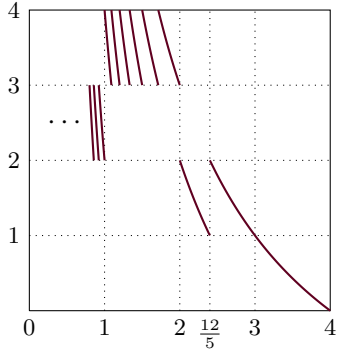
\begin{figure}[h]
\centering
\begin{tikzpicture}
\draw[white] (0,0)--(2,0);
\draw(0,0)node[below]{\small $0$}--(4,0)node[below]{\small $4$}--(4,4)--(0,4)node[left]{\small $4$}--(0,0);
 \draw[dotted] (2.4,0)node[below]{\small $\frac{12}{5}$}--(2.4,4);
 \draw[dotted] (1,0)node[below]{\small $1$}--(1,4);
 \draw[dotted] (2,0)node[below]{\small $2$}--(2,4);
 \draw[dotted] (3,0)node[below]{\small $3$}--(3,4);
 \draw[dotted] (0,1)node[left]{\small $1$}--(4,1);
 \draw[dotted] (0,2)node[left]{\small $2$}--(4,2);
 \draw[dotted] (0,3)node[left]{\small $3$}--(4,3);
 \draw(0.5,2.5)node{\large $\cdots$};

\draw[thick, purple!50!black, smooth, samples =20, domain=2.4:4] plot(\x,{12 / \x -3)});
\draw[thick, purple!50!black, smooth, samples =20, domain=2:2.4] plot(\x,{12 / \x -4)});

\draw[thick, purple!50!black, smooth, samples =20, domain=12/7:2] plot(\x,{12 / \x -3)});
\draw[thick, purple!50!black, smooth, samples =20, domain=12/8:12/7] plot(\x,{12 / \x -4)});
\draw[thick, purple!50!black, smooth, samples =20, domain=12/9:12/8] plot(\x,{12 / \x -5)});
\draw[thick, purple!50!black, smooth, samples =20, domain=12/10:12/9] plot(\x,{12 / \x -6)});
\draw[thick, purple!50!black, smooth, samples =20, domain=12/11:12/10] plot(\x,{12 / \x -7)});
\draw[thick, purple!50!black, smooth, samples =20, domain=1:12/11] plot(\x,{12 / \x -8)});

\draw[thick, purple!50!black, smooth, samples =20, domain=12/13:1] plot(\x,{12 / \x -10)});
\draw[thick, purple!50!black, smooth, samples =20, domain=12/14:12/13] plot(\x,{12 / \x -11)});
\draw[thick, purple!50!black, smooth, samples =20, domain=12/15:12/14] plot(\x,{12 / \x -12)});
\end{tikzpicture}
\caption{The map $T$ of Example 3. The discontinuities are given by $\{ \frac{12}{5}, \frac{12}{6}, \frac{12}{7},\ldots \}$.}
\label{fig:Texample3}
\end{figure}

For $x\in(0,1)$ we find 
\[
x=\frac{12}{\displaystyle d_1+\frac{12}{\displaystyle  d_2+\frac{12}{\ddots}}},
\]
where
\[
d_i(x)\in \begin{cases}
 \mathbb{N}_{\geq10} & \text{ for } i=1\mod 4,\\
  \{3,4,5 \} & \text{ for } i=2\mod 4,\\
   \{3,\ldots, 9\} & \text{ for } i=3 \mod 4,\\
   \{3,4\} & \text{ for } i=0 \mod 4.\\
\end{cases}
\] 
Here $d_i(x)=5$ for $i=2\mod 4$ only when $T^{i-1}(x)=2$,  $d_i(x)=9$ for $i=3\mod 4$ only when $T^{i-1}(x)=1$ and  $d_i(x)=4$ for $i=0\mod 4$ only when $T^{i-1}(x)=3$.

\bigskip
\section{Convergence}\label{sec:convetc}
In this section we prove the convergence for all allowable continued fraction algorithms. We first obtain recurrence relations and several standard equalities and inequalities in the same manner as for the regular continued fraction (by the use of M\"obius transformations).  Throughout this article, $N_k$ will always mean $N_{k\mod m+1}$, where $m$ is the number of intervals of $\Omega$.

\noindent
For any $a,b,c,d\in \mathbb{N}$ we define
\[
\left[
\begin{array}{c c}
 a & b\\
 c& d\\
\end{array}
\right](x)=\frac{ax+b}{cx+d}
\]
and
\[
B_{d,N}=\left[
\begin{array}{c c}
 0 & N\\
 1& d\\
\end{array}
\right].
\]
Furthermore, we define $M_n=B_{d_1,N_1}B_{d_{2},N_{2}}\dots B_{d_n,N_n}$. Note that $\det (M_n)=(-1)^n \prod_{i=1}^n N_i$.
Just as in the classical case we have
\[
M_n=
\left[
\begin{array}{c c}
 p_{n-1}& p_{n}\\
 q_{n-1}& q_{n}\\
\end{array}
\right].
\]
We find that 
\begin{equation}\label{eq:det}
 p_{n-1}q_{n}-q_{n-1} p_{n}=(-1)^n\prod_{i=1}^n N_i . 
\end{equation}
Writing out $M_n=M_{n-1}B_{d_n,N_n}$ gives us the recurrence relations
\begin{eqnarray}
&&p_{-1}=1 \quad p_0=0 \quad p_n=d_np_{n-1}+N_np_{n-2},\\
&&q_{-1}=0 \quad q_0=1 \quad q_n=d_nq_{n-1}+N_nq_{n-2}.\label{req:q}
\end{eqnarray}
Since $x=M_n(T^n(x))$ we can write 
\begin{equation}\label{eq:xfracTk}
x=\frac{p_n+p_{n-1}T^n(x)}{q_n +q_{n-1} T^n(x)}
\end{equation}
and also
\begin{equation}\label{eq:Tn}
T^n(x)=\frac{q_n x -p_n}{-q_{n-1}x+p_{n-1}} .
\end{equation}
 From  (\ref{eq:det}) and (\ref{eq:xfracTk}) we find:
\begin{equation}\label{eq:xminpnqn}
x-\frac{p_n}{q_n}=\frac{(-1)^n N_1\cdots N_nT^n(x)}{q_n(q_n+q_{n-1}T^n(x))} 
\end{equation}
so that
\begin{equation}
|x-\frac{p_n}{q_n}|<\frac{ N_1\cdots N_n \max{\Omega}}{q_n^2}. 
\end{equation}
If we show that 
\[
\lim_{n\rightarrow \infty} \frac{ N_1\cdots N_n }{q_n^2}=0
\]
we have proved convergence. Using the recurrent relation (\ref{req:q}) we find $q_n>(N_n+1)q_{n-2}$ which gives us
\[
 q_n>\begin{cases}
(N_n+1) (N_{n-2}+1) \cdots (N_2+1) & \text{ for } n \text{ even,}\\
(N_n+1) (N_{n-2}+1) \cdots (N_1+1) & \text{ for } n>1 \text{ odd.}\\
\end{cases}
\]
Using this inequality for $q_{n-1}$ and $q_{n-2}$ and substituting this in the recurrent relation gives
\begin{equation}\label{Inequlaity}
 q_n>\begin{cases}
(N_{n-1}+1) (N_{n-3}+1) \cdots (N_1+1)+ N _n(N_{n-2}+1) (N_{n-4}+1) \cdots (N_2+1) & \text{for } n \text{ even,}\\
(N_{n-1}+1) (N_{n-3}+1) \cdots (N_2+1) +N_n(N_{n-2}+1) (N_{n-4}+1) \cdots (N_1+1)  & \text{for } n>1 \text{ odd.}\\
\end{cases}
\end{equation}
We find
\begin{eqnarray*}
 \frac{ N_1\cdots N_n }{q_n^2}&<& \frac{ N_1\cdots N_n }{\Big((N_{n-1}+1) (N_{n-3}+1) \cdots (N_1+1)+ N _n(N_{n-2}+1) (N_{n-4}+1) \cdots (N_2+1) \Big)^2}\\
 &<&   \frac{1 }{\frac{((N_{n-1}+1) (N_{n-3}+1) \cdots (N_1+1))^2}{ N_1\cdots N_n}+2+\frac{( N _n(N_{n-2}+1) (N_{n-4}+1) \cdots (N_2+1))^2}{ N_1\cdots N_n}}\\
  &<&   \frac{1 }{\frac{((N_{n-1}+1) (N_{n-3}+1) \cdots (N_1+1))}{ N_2 N_4\cdots N_n}+2+\frac{( N _n(N_{n-2}+1) (N_{n-4}+1) \cdots (N_2+1))}{ N_1N_3\cdots N_{n-1}}}\\
  &<&   \frac{1 }{\frac{((N_{n-1}+1) (N_{n-3}+1) \cdots (N_1+1))}{ N_2 N_4\cdots N_n}+\frac{( N _n(N_{n-2}+1) (N_{n-4}+1) \cdots (N_2+1))}{ N_1N_3\cdots N_{n-1}}}
\end{eqnarray*}

Suppose  $\textbf{N}=(N_1,\ldots, N_m)$ and $m$ is even ($m$ is odd goes analogous). Now let us write $n=mp+r$ with $0\leq r<m$. Then $N_1\cdots N_n=(N_1\cdots N_m)^p N_1\cdots N_r$ which gives us
\begin{eqnarray*}
   && \frac{1 }{\frac{((N_{n-1}+1) (N_{n-3}+1) \cdots (N_1+1))}{ N_2 N_4\cdots N_n}+\frac{( N _n(N_{n-2}+1) (N_{n-4}+1) \cdots (N_2+1))}{ N_1N_3\cdots N_{n-1}}}\\
  &=& \frac{1 }{\frac{ ( (N_1+1)(N_3 +1)\cdots (N_{m-1}+1))^p (N_1+1)(N_3 +1)\cdots (N_{r-1}+1) }{ (N_2 N_4\cdots N_m)^p N_2 N_4\cdots N_r }+\frac{ ( (N_2+1)(N_4 +1)\cdots (N_m+1))^p (N_2+1)(N_4 +1)\cdots (N_r+1) }{ (N_1 N_3\cdots N_{p-1})^m N_1 N_3\cdots N_{r-1} }}\\
  &<&  \frac{1 }{\frac{ ( (N_1+1)(N_3 +1)\cdots (N_{m-1}+1))^p }{ (N_2 N_4\cdots N_m)^p}C_1+\frac{ ( (N_2+1)(N_4 +1)\cdots (N_m+1))^p }{ (N_1 N_3\cdots N_{m-1})^p  }C_2}\label{ineq:wheneq}\\
  &<&  \frac{1 }{(\frac{ N_1N_3 \cdots N_{m-1}}{ N_2 N_4\cdots N_m)})^pC_1+(\frac{ ( N_2N_4\cdots N_m }{ (N_1 N_3\cdots N_{m-1})  })^pC_2}\label{ineq:whenineq}
\end{eqnarray*}
for some $C_1,C_2\in \mathbb{R}$. Now note that as $n\rightarrow \infty$ we have that $p\rightarrow\infty$ so that whenever $N_1N_3 \cdots N_{m-1}\neq N_2 N_4\cdots N_m$  the last inequality goes to $0$. Whenever $N_1N_3 \cdots N_{m-1}= N_2 N_4\cdots N_m$ then we have
\begin{eqnarray}
&& \frac{1 }{\frac{ ( (N_1+1)(N_3 +1)\cdots (N_{m-1}+1))^p }{ (N_2 N_4\cdots N_m)^p}C_1+\frac{ ( (N_2+1)(N_4 +1)\cdots (N_m+1))^p }{ (N_1 N_3\cdots N_{m-1})^p  }C_2}\\
 &<& \frac{1 }{K_1^m C_1+K_2^m C_2}
\end{eqnarray}
for some $K_1,K_2>1$ and thus we find convergence also in that case.

\bigskip
\section{Ergodicity for the allowable case}\label{sec:ergodic}

Here we prove ergodicity (and the existence of an absolutely continuous invariant measure) when we are in the allowable case. We use the following result of Zweim\"uller (see~\cite{Z2020}). Although the theorem is stated when the underlying space is the unit interval,  the results hold true if $[0,1]$ is replaced by a finite disjoint union of intervals. Throughout the paper, we denote normalized Lebesgue measure by $\lambda$.
\begin{theorem}(Zweim\"uller)\label{zweimullar}
Let $T : [0, 1] \to  [0, 1]$, and let B be a collection (not necessarily finite) of nonempty pairwise disjoint open subintervals with $\lambda(\bigcup B)= 1$ such that $T$ restricted to each element $Z$ of $B$ is continuous and strictly monotone. Suppose $T$ satisfies the following three conditions:
\begin{itemize}
\item[(A)] Adler’s condition: $T^{\prime\prime}/(T^{\prime})^2$ is bounded on $\bigcup B$,
\item[(B)] Finite image condition: $TB=\{TZ:Z\in B\}$ is finite,
\item[(C)] Uniformly eventually expanding: there is a $k$ such that $|(T^k)^{\prime}|>\gamma>1$ on $\bigcup B$.
\end{itemize}
Then there are a finite number of pairwise disjoint open sets $X_1,\cdots, X_n$ such that $TX_i = X_i$ (modulo sets of $\lambda$-measure zero) and $T|_{X_i}$ is conservative and ergodic with respect to $\lambda$.
Almost all points of $[0, 1]\setminus \bigcup_i X_i$ are eventually mapped into one of these ergodic components and every ergodic component $X_i $ can be written as a finite union of open intervals.
Furthermore, each $X_i$ supports an absolutely continuous invariant measure $\mu_i$ which is unique up to a constant factor.
\end{theorem}

\begin{theorem}
Let $\Omega=\bigcup_{i=1}^{m} I_i$ where $I_i:=[a_i,a_i+1)$, with $a_i\in\mathbb{N}\cup \{0\}$, are distinct intervals. Assume $\textbf{N}=(N_1,\ldots, N_{m})\in \mathbb{N}^m$ is allowable for $\Omega$ and $T:\Omega \to\Omega$ as given in equation (\ref{definitionT}). Then, $T$ admits a unique absolutely continuous invariant ergodic measure.
\end{theorem}

\begin{proof}
We apply Theorem (\ref{zweimullar}).
For our setup, we take $B$ to be the collection of all fundamental intervals of rank 1. Since $\Omega$ consists of $m$ disjoint intervals and on each each interval $I_i$, there are at most two fundamental intervals (corresponding to the smallest and largest digit on that interval) that might not be mapped onto $I_{i+1}$, we see that $TB$ consists of at most $2m$ elements. Hence $T$ satisfies condition $(B)$.
We only need to check conditions $(A)$ and $(C)$ since the rest is clearly satisfied.
Using (\ref{eq:det}) and  (\ref{eq:Tn}) we find 
\begin{equation}\label{eq:derTk}
(T^k)^\prime(x)=\frac{(-1)^k  \prod_{i=1}^k N_i}{(p_{k-1}-q_{k-1}x)^2}.
\end{equation}
We will first prove (C) and find a $k,
 \gamma$ such that $\left|(T^k(x))^\prime\right|>\gamma>1 $.
 For (A) we remark that $T$ can be replaced by $T^k$ and the theorem will still hold. In the second part of the proof we find an $M>0$ such that $\left|\frac{(T^k(x))^{\prime \prime}}{((T^k(x))^\prime)^2}\right|\leq M$. 
We will assume with no loss of generality that $x\in I_1$ since the proof for $x\in I_i$ follows the same arguments.
Note that by definition we have 
\[
T^{k+1}(x)=\frac{N_{k+1}}{T^k(x)}-d_{k+1}
\]
which gives
\begin{equation}\label{eq:TninTnplus1}
T^k(x)=\frac{N_{k+1}}{d_{k+1}+T^{k+1}(x)}.
\end{equation}
Substituting (\ref{eq:TninTnplus1}) in (\ref{eq:xminpnqn}) gives
\begin{equation}
x-\frac{p_k}{q_k}=\frac{(-1)^k N_1\cdots N_k \frac{N_{k+1}}{d_{k+1}+T^{k+1}(x)}}{q_k(q_k+q_{k-1}\frac{N_{k+1}}{d_{k+1}+T^{k+1}(x)})}= \frac{(-1)^k N_1\cdots N_k N_{k+1}}{q_k(q_k(d_{k+1}+T^{k+1}(x))+N_{k+1}q_{k-1})}.
\end{equation}
Using the recurrence relation (\ref{req:q}) we find
\begin{equation}\label{eq:xminpkqk}
x-\frac{p_k}{q_k}=\frac{(-1)^k N_1\cdots N_k N_{k+1}}{q_k(q_{k+1} +q_kT^{k+1}(x))}.
\end{equation}
This, and (\ref{req:q})  gives us 
\begin{equation}
\frac{N_1\cdots N_k N_{k+1}}{q_kq_{k+1}(1+\max \Omega)}<\left|x-\frac{p_k}{q_k}\right|<\frac{N_1\cdots N_n N_{k+1}}{q_kq_{k+1}}.
\end{equation}
This gives us 
\begin{eqnarray}
&&\frac{N_1\cdots N_k N_{k+1}}{q_{k+1}(1+\max \Omega)}<\left|q_kx-p_k\right|<\frac{N_1\cdots N_k N_{k+1}}{q_{k+1}}\label{ineq}\\
&&\frac{N_1^2\cdots N_k^2 N_{k+1}^2}{q_{k+1}^2(1+\max \Omega)^2}<\left(q_kx-p_k\right)^2<\frac{N_1^2\cdots N_k^2 N_{k+1}^2}{q_{k+1}^2}.\label{qnxminpn}
\end{eqnarray}
From (\ref{eq:derTk}) and (\ref{qnxminpn}) it now follows that 

\begin{equation}
|(T^k)^\prime(x)| > \frac{q_{k}^2   \prod_{i=1}^k N_i }{N_1^2\cdots N_{k-1}^2 N_{k}^2}=\frac{q_{k}^2}{N_1\cdots N_{k-1} N_{k}}.
\end{equation}
Using that $(a+b)^2\geq 4ab$ and equation (\ref{Inequlaity}), we find for $k\geq 2$ that $q_k^2>4 \prod_{i=1}^k N_i$ for both the odd and the even case. Using this in (\ref{qnxminpn}) we find:
\begin{equation}
|(T^k)^\prime(x)| > 4 >1.
\end{equation}
Now we will prove (A), i.e.  $\left|\frac{(T^k)^{\prime \prime}(x)}{((T^k)^\prime(x))^2}\right|\leq M$ for some $M>0$. Using (\ref{eq:derTk}) we find 
\begin{equation}\label{eq:secderTk}
(T^k)^{\prime \prime}(x)=\frac{(-1)^{k}2 q_{k-1}   \prod_{i=1}^k N_i}{(p_{k-1}-q_{k-1}x)^3}.
\end{equation}
Filling (\ref{eq:secderTk}) and (\ref{eq:derTk}) in $\left|\frac{(T^k)^{\prime \prime}(x)}{((T^k)^\prime(x))^2}\right|$ gives
\begin{equation}
\left | \frac{(T^k)^{\prime \prime}(x)}{((T^k)^\prime(x))^2}\right| =\frac{2q_{k-1}  |p_{k-1}-q_{k-1} x|}{\prod_{i=1}^k N_i}.
\end{equation}
Using (\ref{ineq}) we find
\[
\frac{2q_{k-1}  |p_{k-1}-q_{k-1} x|}{\prod_{i=1}^k N_i}<\frac{2q_{k-1}  \prod_{i=1}^k N_i }{q_k \prod_{i=1}^k N_i}<2.
\]
From the definition of our map $T$, it is easy to see that $\Omega$ is the smallest forward invariant set. From this we can conclude from Theorem (\ref{zweimullar}) that $T$ has a unique invariant ergodic measure $\mu_T$ that is absolutely continuous with respect to the Lebesgue measure $\lambda$.
\end{proof}

\bigskip
\section{The natural extension and a  Doeblin-Lenstra type theorem}\label{sec:naturalextension}
 
In this section we study the natural extension of the map $T$ defined above, and we use it to prove a Doeblin-Lenstra type theorem. We will give the natural extension in the simple case when we have two intervals. In the next section we discuss the possibilities and impossibilities of other cases.

\begin{theorem}\label{th:natex}
For a simple system $(T,\Omega,\mu_T,\mathcal{B}_\Omega)$ with $\Omega=[a_1,a_1+1)\cup [a_2,a_2+1)$ we define $X=[a_1,a_1+1)\times [a_2,a_2+1]\cup [a_2,a_2+1)\times [a_1,a_1+1]$. The natural extension is given by $(\overline{\mathcal{T}_T},X,\overline{\mu_T},\mathcal{B}_X)$ where $\overline{\mathcal{T}_T}:X\rightarrow X$ is defined as 
\[
\overline{\mathcal{T}_T}(x,y)=\left(T(x),\frac{N}{d_1(x)+y}\right)
\]
Furthermore, the invariant measure $\overline{\mu_T}$ is given by
\[
\overline{\mu_T}(A)=C \iint\limits_A \frac{N}{(N+xy)^2} dy dx,
\]
where $C^{-1}=\iint\limits_X \frac{N}{(N+xy)^2} dy dx=2 \ln\Big[1+\frac{N}{(N+a_1(a_2+1))(N+a_2(a_1+1))}\Big]$.
\end{theorem}

\begin{proof}
We first show that $\overline{\mathcal{T}_T}$ is bijective almost everywhere. Let $\{c_1,\ldots,c_m\}$ be the digit set on $(a_1,a_1+1]$ and $\{d_1,\ldots,d_k\}$ the digit set on  $(a_2,a_2+1]$. The case that the digit set is infinite is in analogy. For the rectangles of the form $\Delta(c_i)\times [a_2,a_2+1]$ we find  $\overline{\mathcal{T}_T}(\Delta(c_i)\times [a_2,a_2+1])=[a_2,a_2+1]\times [\frac{N}{c_i+a_2+1}, \frac{N}{c_i+a_2}]$. On these rectangles  $\overline{\mathcal{T}_T}$ is bijective. It is easy to check that these rectangles fill up $[a_2,a_2+1]\times [a_1,a_1+1]$ and only overlap on the edges. In the same way the images of the rectangles $\Delta(d_i)\times [a_1,a_1+1]$ fill up the rectangle  $[a_1,a_1+1]\times [a_2,a_2+1]$, see also Figure~\ref{fig:natex2int}. 

\begin{figure}[h]
\centering
\begin{tikzpicture}[scale=2]

\draw(1,3)[white]node[below, black]{\small $a_1$}node[left, black]{\small $a_2$}--(2,3)node[below, black]{\small $a_1+1$}--(2,4)--(1,4)node[left, black]{\small $a_2+1$}--(1,3);
\draw(1,3)rectangle(1.3,4)node[midway]{\tiny $\Delta(c_m)$};
\draw(1.3,3)rectangle(1.6,4)node[midway]{\tiny $\cdots$};
\draw(1.6,3)rectangle(2,4)node[midway]{\small $\Delta(c_1)$};

\draw(3,1)[white]node[below, black]{\small $a_2$}node[left, black]{\small $a_1$}--(4,1)node[below, black]{\small $a_2+1$}--(4,2)--(3,2)node[left, black]{\small $a_1+1$}--(3,1);
\draw(3,1)rectangle(3.3,2)node[midway]{\tiny $\Delta(d_k)$};
\draw(3.3,1)rectangle(3.6,2)node[midway]{\tiny $\cdots$};
\draw(3.6,1)rectangle(4,2)node[midway]{\small $\Delta(d_1)$};

\draw[-latex,line width=1mm] (3.5,2.5) -- (4.5,2.5)node[midway, above]{\large $\mathcal{T}_T$};

\draw(4.5,3)[white]node[below, black]{\small $a_1$}node[left, black]{\small $a_2$}--(5.5,3)node[below, black]{\small $a_1+1$}--(5.5,4)--(4.5,4)node[left, black]{\small $a_2+1$}--(4.5,3);
\draw(4.5,3)rectangle(5.5,3.3)node[midway]{\small $\mathcal{T}_T(\Delta(d_k))$};
\draw(4.5,3.3)rectangle(5.5,3.7)node[midway]{\small $\vdots$};
\draw(4.5,3.7)rectangle(5.5,4)node[midway]{\small $\mathcal{T}_T(\Delta(d_1))$};

\draw(6.5,1)[white]node[below, black]{\small $a_2$}node[left, black]{\small $a_1$}--(7.5,1)node[below, black]{\small $a_2+1$}--(7.5,2)--(6.5,2)node[left, black]{\small $a_1+1$}--(6.5,1);
\draw(6.5,1)rectangle(7.5,1.3)node[midway]{\small $\mathcal{T}_T(\Delta(c_m))$};
\draw(6.5,1.3)rectangle(7.5,1.7)node[midway]{\small $\vdots$};
\draw(6.5,1.7)rectangle(7.5,2)node[midway]{\small $\mathcal{T}_T(\Delta(c_1))$};

\end{tikzpicture}
\caption{The domain of the natural extension and its image under $\overline{\mathcal{T}_T}$ in the case of two intervals.}
\label{fig:natex2int}
\end{figure}
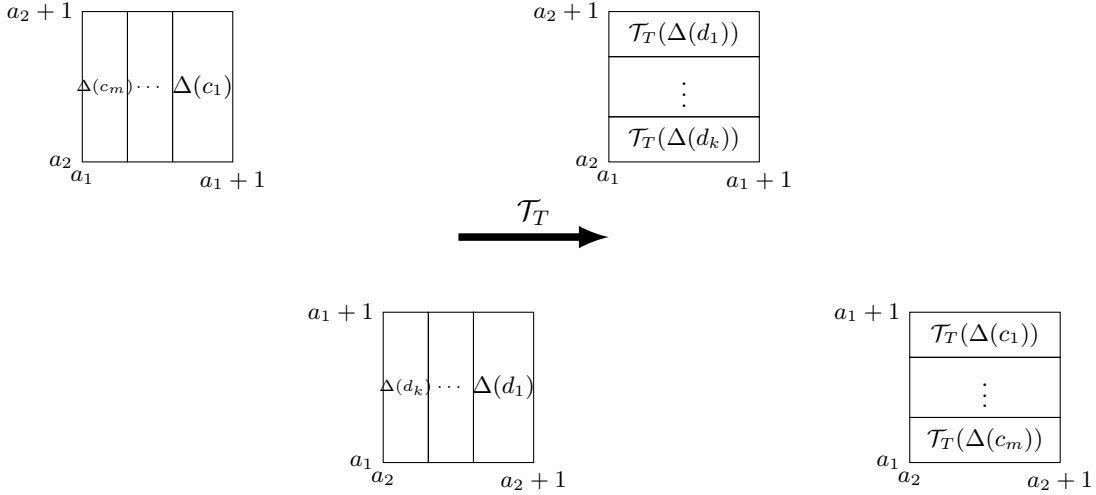
The proof that $\mathcal{B}_X=\bigvee_{i=0}^\infty \mathcal{T}^i_T(\mathcal{B}_\Omega \times \Omega )$ is in analogy with~\cite{DKW13}. Using Jacobians as in~\cite{DKW13} one finds that $\overline{\mu_T}(A)=C\iint\limits_A \frac{N}{(N+xy)^2} dy dx$ is indeed  $\overline{\mathcal{T}_T}$-invariant.
\end{proof}

It  is worth mentioning that dynamical systems for lazy $N$-expansions from~\cite{DKW13} are examples of such systems.  In Section \ref{othercases}, we show that for the non-simple but allowable case of two intervals, one is still able to build the domain of the natural extension, however we were not able to determine the density of the invariant measure.
An immediate consequence of Theorem~\ref{th:natex} is the following.
\begin{corollary}
Suppose $(T,\Omega, \mu_T, \mathcal{B}_\Omega)$
 is a simple system with $\Omega=[a_1,a_1+1)\cup [a_2,a_2+1)$ then the invariant measure $\mu_T$ is given by
\[
\mu_T(A)= C \int_A \left(\frac{a_2+1}{N+(a_2+1)x}-\frac{a_2}{N+a_2x}\right)\textbf{1}_{[a_1,a_1+1)} +\left(\frac{a_1+1}{N+(a_1+1)x}-\frac{a_1}{N+a_1x}\right)\textbf{1}_{[a_2,a_2+1)} dx
\]
for $A\in \mathcal{B}_\Omega$ and $C$ as in Theorem~\ref{th:natex}.
\end{corollary}
This is found by simply projecting down to the first coordinate. \newline

In the remainder of this section we will look at how well convergents of a number $x$ are approximating $x$. We will first look at the desirable case with two intervals. For $x\in\Omega$ we define:
 \begin{equation}
 \theta_n(x)=\frac{q_n^2}{\prod_{i=1}^n N_i}\left|x-\frac{p_n}{q_n}\right|.
 \end{equation} 
By writing $t_n=T^n(x)$ and $v_n=\frac{N_n q_{n-1}}{q_n}$ by using (\ref{eq:xminpkqk}) we see that we can write $ \theta_n(x)=\frac{N_nt_n}{N_n+t_n v_n}$.  We also find $ \theta_{n-1}(x)=\frac{N_nv_n}{N_n+t_n v_n}$. Furthermore, note that $\mathcal{T}^n(x,0)=(t_n,v_n)$. Now since $\mathcal{T}^n(x,0)\in I_1\times I_2 \cup I_2\times I_1$ for all $n$ we find the following estimates for $\theta_n(x)$:
\begin{eqnarray}
\frac{a_1 N_n}{N_n+a_1(a_2+1)} \leq \theta_n(x) \leq \frac{N_n(a_1+1)}{N_n+a_2(a_1+1)}\quad  \text{ if } t_n\in I_1,\label{ineq:thetha1}\\
\frac{a_2 N_n}{N_n+a_2(a_1+1)} \leq \theta_n(x) \leq \frac{N_n(a_2+1)}{N_n+a_1(a_2+1)}\quad  \text{ if } t_n\in I_2.\label{ineq:thetha2}
\end{eqnarray}
When we are in the simple case we can say more. We formulate a  Doeblin-Lenstra type theorem.
For $c\in\mathbb{R}$, let $A_c:= \{(x,y)\in X: \frac{Nx}{N+xy}\leq c\}$. 
\begin{theorem}
Let $F(c):=\overline{\mu_T}(A_c)$ and let $C$ be the normalizing constant. Then we have
\[
F(c)=\begin{cases}
0 & \text{for } c \in \left(-\infty,\frac{Na_1}{N+a_1(a_2+1)}\right]\\
 C \left( \frac{c(a_2+1+\frac{N}{a_1})}{N}-1-\log\left (\frac{Nc+ca_1(a_2+1)}{a_1N}\right)  \right) & \text{for } c\in \left(\frac{Na_1}{N+a_1(a_2+1)},\frac{Na_1}{N+a_1a_2}\right]\\
C  \left( - \log \left( \frac{N+a_1(a_2+1)}{N+a_1a_2} \right)+ \frac{c}{N} \right) & \text{for } c\in \left(\frac{Na_1}{N+a_1a_2},\frac{N(a_1+1)}{N+(a_1+1)(a_2+1)}\right]\\
C  \left(\log \left( \frac{Nc+c(a_1+1)(a_2+1)}{N(a_1+1)}\frac{N+a_1a_2}{N+a_1(a_2+1)}  \right)+ 1-\frac{c}{a_1+1}-\frac{ca_2}{N}\right) & \text{for } c\in \left(\frac{N(a_1+1)}{N+(a_1+1)(a_2+1)},\frac{N(a_1+1)}{N+a_2(a_1+1)}\right]\\
\frac{1}{2} & \text{for } c\in \left(\frac{N(a_1+1)}{N+(a_1+1)(a_2+1)},\frac{Na_2}{N+a_2(a_1+1)}\right]\\
\frac{1}{2}+C\left( \frac{c(a_1+1+\frac{N}{a_2})}{N}-1-\log\left (\frac{Nc+ca_2(a_1+1)}{a_2N}\right)  \right) & \text{for } c\in \left(\frac{Na_2}{N+a_2(a_1+1)},\frac{Na_2}{N+a_1a_2}\right]\\
\frac{1}{2} +C  \left( - \log \left( \frac{N+a_2(a_1+1)}{N+a_1a_2} \right)+ \frac{c}{N} \right) & \text{for } c\in \left(\frac{Na_2}{N+a_1a_2},\frac{N(a_2+1)}{N+(a_1+1)(a_2+1)}\right]\\
 \frac{1}{2} + C \left( \log \left( \frac{Nc+c(a_2+1)(a_1+1)}{N(a_2+1)}\frac{N+a_1a_2}{N+a_2(a_1+1)}  \right)+ 1-\frac{c}{a_2+1}-\frac{ca_1}{N}  \right) & \text{for } c\in \left(\frac{N(a_2+1)}{N+(a_1+1)(a_2+1)},\frac{N(a_2+1)}{N+a_1(a_2+1)}\right]\\
1 & \text{for } c\in \left(\frac{N(a_2+1)}{N+a_1(a_2+1)},\infty\right)\\
\end{cases}
\]
\end{theorem}

\begin{proof}
An argument similar to the one given by Jager (\cite{DK2002} Lemma 5.3.11) shows that for almost all $x$, the orbit of $(x,0)$ under $\overline{{\mathcal T}_T}$ is a generic point under $\overline{\mu_T}$, hence, by 
the ergodic theorem, we have for almost all $x$,
\[
\lim_{n\rightarrow\infty} \frac{1}{n}\{1\leq i\leq n : \theta_i(x)\leq c\}=\lim_{n\to\infty}\frac{1}{n}\sum_{i=0}^n {\mathbb I}_{A_c}({\overline{\mathcal T}_T}^i(x,0))=\overline{\mu_T}(A_c).
\]
The rest of the proof consists of calculating a list of  integrals. Figure~\ref{fig:DLbounds} illustrates where the formula of $F(c)$ changes. Note that in Figure~\ref{fig:DLbounds}  we have $
\frac{Na_1}{N+a_1a_2}=\frac{N(a_1+1)}{N+(a_1+1)(a_2+1)}$. This only happens when $N=a_2(a_2+1)$. The minimal $N$ for a chosen $\Omega$ to make the system simple is at least equal to $a_2(a_2+1)$.

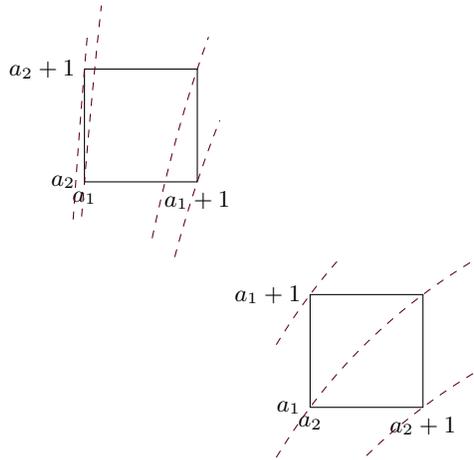
\begin{figure}[h]
\centering
\begin{tikzpicture}[scale=1.5]

\draw(1,3)node[below, black]{\small $a_1$}node[left, black]{\small $a_2$}--(2,3)node[below, black]{\small $a_1+1$}--(2,4)--(1,4)node[left, black]{\small $a_2+1$}--(1,3);
\draw [dashed, purple!50!black, smooth, samples =20, domain=0.9:1.025] plot(\x,{16 -12/\x)});
\draw [dashed, purple!50!black, smooth, samples =20, domain=0.975:1.15] plot(\x,{15 -12/\x)});
\draw [dashed, purple!50!black, smooth, samples =20, domain=1.6:2.1] plot(\x,{10 -12/\x)});
\draw [dashed, purple!50!black, smooth, samples =20, domain=1.8:2.2] plot(\x,{9 -12/\x)});

\draw(3,1)node[below, black]{\small $a_2$}node[left, black]{\small $a_1$}--(4,1)node[below, black]{\small $a_2+1$}--(4,2)--(3,2)node[left, black]{\small $a_1+1$}--(3,1);

\draw [dashed, purple!50!black, smooth, samples =20, domain=2.7:3.25] plot(\x,{6 -12/\x)});
\draw [dashed, purple!50!black, smooth, samples =20, domain=2.7:4.5] plot(\x,{5 -12/\x)});
\draw [dashed, purple!50!black, smooth, samples =20, domain=3.5:4.5] plot(\x,{4 -12/\x)});

\end{tikzpicture}
\caption{The domain of the natural extension with dashed curves of the form $\frac{Nx}{N+xy}=c$ for the values where the formula for $F(c)$ changes.}
\label{fig:DLbounds}
\end{figure}
\noindent 1) For $ c \in \left(-\infty,\frac{Na_1}{N+a_1(a_2+1)}\right]$ we know that $A_c=\emptyset$ so $F(c)=0$.\\
2) For $c\in \left(\frac{Na_1}{N+a_1(a_2+1)},\frac{Na_1}{N+a_1a_2}\right]$:
\begin{eqnarray*}
 F(c)&=&
 C\int \displaylimits_{\frac{N(a_1-c)}{ca_1}}^{a_2+1} \int \displaylimits_{a_1}^{\frac{Nc}{N-cy}} \frac{N}{(N+xy)^2} dx dy =
 C\int \displaylimits_{\frac{N(a_1-c)}{ca_1}}^{a_2+1}  \frac{\frac{Nc}{N-cy}}{N+(\frac{Nc}{N-cy})y}-  \frac{a_1}{N+a_1y}   dy\\
 &=&
 C\int \displaylimits_{\frac{N(a_1-c)}{ca_1}}^{a_2+1}  \frac{Nc}{N(N-cy)+Ncy}-  \frac{a_1}{N+a_1y}   dy=
 C\int \displaylimits_{\frac{N(a_1-c)}{ca_1}}^{a_2+1}  \frac{c}{N}-\frac{a_1}{N+a_1y}   dy\\
&=&
 C \left(  \frac{c(a_2+1)}{N}- \frac{c(\frac{N(a_1-c)}{ca_1})}{N}-\log\left( \frac{N+a_1(a_2+1)}{N+a_1\frac{N(a_1-c)}{ca_1}}\right) \right)\\
 &=&
 C \left( \frac{c(a_2+1+\frac{N}{a_1})}{N}-1-\log\left (\frac{Nc+ca_1(a_2+1)}{a_1N}\right)  \right).\\
\end{eqnarray*}
3) If $N=a_2(a_2+1)$ then $\frac{Na_1}{N+a_1a_2}= \frac{N(a_1+1)}{N+(a_1+1)(a_2+1)} $. Else, for $c\in \left(\frac{Na_1}{N+a_1a_2},\frac{N(a_1+1)}{N+(a_1+1)(a_2+1)}\right]$:

\begin{eqnarray*}
F(c)&=&C\int \displaylimits_{a_2}^{a_2+1} \int \displaylimits_{a_1}^{\frac{Nc}{N-ca_2}} \frac{N}{(N+xy)^2} dx dy + C\int \displaylimits_{a_2}^{a_2+1} \int \displaylimits_{\frac{Nc}{N-ca_2}}^{\frac{Nc}{N-cy}} \frac{N}{(N+xy)^2} dx dy \\
&=&
C\int \displaylimits_{a_2}^{a_2+1} \frac{(\frac{Nc}{N-ca_2})}{N+(\frac{Nc}{N-ca_2})y}-\frac{a_1}{N+a_1y}  +   \frac{(\frac{Nc}{N-cy})}{N+(\frac{Nc}{N-cy})y}-\frac{(\frac{Nc}{N-ca_2})}{N+(\frac{Nc}{N-ca_2})y} dy\\
&=&
C\int \displaylimits_{a_2}^{a_2+1} \frac{Nc}{(N-ca_2)N+Ncy}-\frac{a_1}{N+a_1y}  +  \frac{Nc}{(N-cy)N+Ncy}-\frac{Nc}{(N-ca_2)N+Ncy} dy\\
&=& 
C  \left( - \log (N+a_1y) + \frac{cy}{N} |_{a_2}^{a_2+1}     \right)= C  \left( - \log \left( \frac{N+a_1(a_2+1)}{N+a_1a_2} \right)+ \frac{c}{N} \right).\\
\end{eqnarray*}
4) For $c\in \left(\frac{N(a_1+1)}{N+(a_1+1)(a_2+1)},\frac{N(a_1+1)}{N+a_2(a_1+1)}\right]$:
\begin{eqnarray*}
F(c)&=& C\int \displaylimits_{\frac{N(a_1+1-c)}{c(a_1+1)}}^{a_2+1} \int \displaylimits_{a_1}^{a_1+1} \frac{N}{(N+xy)^2} dx dy +C\int \displaylimits_{a_2}^{\frac{N(a_1+1-c)}{c(a_1+1)}} \int \displaylimits_{a_1}^{\frac{Nc}{N-cy}} \frac{N}{(N+xy)^2} dx dy \\
&=&
C\int \displaylimits_{\frac{N(a_1+1-c)}{c(a_1+1)}}^{a_2+1} \frac{a_1+1}{N+(a_1+1)y}-\frac{a_1}{N+a_1y} dy +C\int \displaylimits_{a_2}^{\frac{N(a_1+1-c)}{c(a_1+1)}} \frac{\frac{Nc}{N-cy}}{N+(\frac{Nc}{N-cy})y}-\frac{a_1}{N+a_1y} dy\\
&=&
C  \left( \log \left( \frac{N+(a_1+1)(a_2+1)}{N+(a_1+1)(\frac{N(a_1+1-c)}{c(a_1+1)})} \right)-\log \left( \frac{N+a_1(a_2+1)}{N+a_1\frac{N(a_1+1-c)}{c(a_1+1)}} \right)\right.\\ 
&+&
\left.\frac{c}{N}(\frac{N(a_1+1-c)}{c(a_1+1)}-a_2) -\log \left( \frac{N+a_1\frac{N(a_1+1-c)}{c(a_1+1)}}{N+a_1a_2} \right)\right)\\
\end{eqnarray*}

\begin{eqnarray*}
&=&
C  \left(\log \left( \frac{Nc+c(a_1+1)(a_2+1)}{N(a_1+1)}\frac{N+a_1a_2}{N+a_1(a_2+1)}  \right)+ 1-\frac{c}{a_1+1}-\frac{ca_2}{N}\right). \\
\end{eqnarray*}
5) For $c\in \left(\frac{N(a_1+1)}{N+(a_1+1)(a_2+1)},\frac{Na_2}{N+a_2(a_1+1)}\right]$ we know that $A_c=[a_1,a_1+1)\times [a_2,a_2+1]$ so that $F(c)=\frac{1}{2}$.\\
6) For $ c\in \left(\frac{Na_2}{N+a_2(a_1+1)},\frac{Na_2}{N+a_1a_2}\right]$ we have the same calculation as for (2) but then the roles of $a_1$ and $a_2$ are swapped. Therefore we find:
\[
 F(c)=\frac{1}{2} +C\int \displaylimits_{\frac{N(a_2-c)}{ca_2}}^{a_1+1} \int \displaylimits_{a_2}^{\frac{Nc}{N-cy}} \frac{N}{(N+xy)^2} dx dy  =\frac{1}{2}+C\left( \frac{c(a_1+1+\frac{N}{a_2})}{N}-1-\log\left (\frac{Nc+ca_2(a_1+1)}{a_2N}\right)  \right).
\]
7) For $  c\in \left(\frac{Na_2}{N+a_1a_2},\frac{N(a_2+1)}{N+(a_1+1)(a_2+1)}\right]$ we have the same calculation as for (3) but then the roles of $a_1$ and $a_2$ are swapped. Therefore we find:
\begin{eqnarray*}
F(c)&=&\frac{1}{2} +C\int \displaylimits_{a_1}^{a_1+1} \int \displaylimits_{a_2}^{\frac{Nc}{N-ca_1}} \frac{N}{(N+xy)^2} dx dy + C\int \displaylimits_{a_1}^{a_1+1} \int \displaylimits_{\frac{Nc}{N-ca_1}}^{\frac{Nc}{N-cy}} \frac{N}{(N+xy)^2} dx dy\\
&=&\frac{1}{2} +C  \left( - \log \left( \frac{N+a_2(a_1+1)}{N+a_1a_2} \right)+ \frac{c}{N} \right).
\end{eqnarray*}

8) For $ c\in \left(\frac{N(a_2+1)}{N+(a_1+1)(a_2+1)},\frac{N(a_2+1)}{N+a_1(a_2+1)}\right]$  we have the same calculation as for (4) but then the roles of $a_1$ and $a_2$ are swapped. Therefore we find:
\begin{eqnarray*}
F(c)&=&\frac{1}{2} +C\int \displaylimits_{\frac{N(a_2+1-c)}{c(a_2+1)}}^{a_1+1} \int \displaylimits_{a_2}^{a_2+1} \frac{N}{(N+xy)^2} dx dy + C\int \displaylimits_{a_1}^{\frac{N(a_2+1-c)}{c(a_2+1)}} \int \displaylimits_{a_2}^{\frac{Nc}{N-cy}} \frac{N}{(N+xy)^2} dx dy  \\
&=&
 \frac{1}{2} + C \left( \log \left( \frac{Nc+c(a_2+1)(a_1+1)}{N(a_2+1)}\frac{N+a_1a_2}{N+a_2(a_1+1)}  \right)+ 1-\frac{c}{a_2+1}-\frac{ca_1}{N}  \right).
\end{eqnarray*}
9) For $ c\in \left(\frac{N(a_2+1)}{N+a_1(a_2+1)},\infty\right)$ we have $A_c=X$ so that $F(c)=1$.
\end{proof}

\section{Other cases}\label{othercases}
In this section we shed a light on the cases where we did not find the natural extension. We will highlight where the difficulties lie. In the desirable but non-simple case we do not know the density. Calculations show that in the case of two intervals, whenever $N_1\neq N_2$, a measure with density of the form $\frac{K}{(K+xy)^2}$ for some $K\in \mathbb{R}$ is never invariant. Therefore, there is no easy candidate for the invariant measure. We do like to point out that the domain as in Theorem \ref{th:natex} is still Lebesgue almost bijective in case there are two intervals. Let us now focus on the simple case on more than two intervals. We have the following conjecture which we support by numerical analyses. We will show that the algorithm we used converges for two intervals. Furthermore, we will compare the results of the algorithm with a known and an unknown density.

\begin{conjecture}\label{con:natex}
Let $(T,\Omega,\mu_T,\mathcal{B}_\Omega)$ be a simple system that is ergodic with $\mu_T$ an a.c.i.m., and $\mathcal{B}_\Omega$ the Borel $\sigma$-algebra on $\Omega$. Let  $\mathcal{T}_T:\Omega \times [0,\infty) \rightarrow \Omega \times [0,\infty) $ be defined as 
\[
\mathcal{T}_T(x,y)=\left(T(x),\frac{N}{d_1(x)+y}\right).
\]
Now let $X=\cap_{i=0}^\infty \mathcal{T}_T^i(\Omega \times [0,\infty) )$. Then 
\[
\overline{\mathcal{T}_T}:=\mathcal{T}_T|_X
\]
is bijective Lebesgue almost everywhere and $(\overline{\mathcal{T}_T},X,\overline{\mu_T},\mathcal{B}_X)$ is the natural extension of  $(T,\Omega,\mu_T, \mathcal{B}_\Omega )$ where $\overline{\mu_T}$ is an invariant measure that is absolutely continuous with respect to the two dimensional Lebesgue measure given by
\[
\overline{\mu_T}(A)=C\iint\limits_A  \frac{N}{(N+xy)^2} dy dx
\] 
with $A\in \mathcal{B}_X$ where $\mathcal{B}_X$  is the Borel $\sigma$-algebra restricted to $X$ and $C$ a normalising constant.
\end{conjecture}
\begin{proof}[partial proof]
Let $\mu_T^*(A)$ be $\overline{\mu_T}(A)$ but on the Borel $\sigma$-algebra restricted to $\Omega \times [0,\infty)$. We first want to show that 
\[
0<\mu_T^*(X)<\infty.
\]
For $\mu_T^*(X)<\infty$ note that $\mathcal{T}_T\left(\Omega \times [0,\infty)
\right)=I_1\times [0,\frac{N}{d_1(a_m)}) \cup \cup_{i=2}^m (I_i \times [0,\frac{N}{d_1(a_{i-1})})$ so that $\mu_T^*\left( \mathcal{T}_T(\Omega \times [0,\infty)
 \right)<\infty$. Since $X\subset \mathcal{T}_T\left(\Omega \times [0,\infty)
\right)$ we find $\mu_T^*(X)<\infty$. For $0<\mu_T^*(X)$ we want to show that $X$ contains a rectangle. Without loss of generality we can assume this rectangle is a subset of $I_1\times [0,\infty)$. The line ${y}\times I_1$ is contained in $X$ if we can write $y=[0;d_1,d_2,\ldots]_N$ where $d_1$ comes from the digit set corresponding to the interval $I_m$, $d_2$ from the digit set corresponding to the interval $I_{m-1}$ etc.. Note that the order of possible digits is reversed! The question whether $X$ contains a rectangle now translates to whether the set of all such $y$'s contains an interval. \textbf{This is where the gap in the proof is as we were unable to prove that $X$ has positive Lebesgue measure.}

\vspace{1em}

For surjectivity of $\overline{\mathcal{T}_T}$. Since $\mathcal{T}_T(\Omega \times [0,\infty) )\subset  \Omega \times [0,\infty)$ we have that $\mathcal{T}_T(X)=\cap_{i=1}^\infty \mathcal{T}^i(\Omega \times [0,\infty) )=\cap_{i=0}^\infty \mathcal{T}^i(\Omega \times [0,\infty) )=X$.  We prove injectivity using contradiction. Suppose $\overline{\mathcal{T}_T}$ is not injective on $X$. Since  $\overline{\mathcal{T}_T}$ is surjective we  have that for every $(x_2,y_2)\in X$ there is at least $(x_1,y_1)\in X$  such that $ \overline{\mathcal{T}_T}(x_1,y_1)=(x_2,y_2)$. Now let $Y$ be the set of $(x_2,y_2)\in X$ such that there are at least two pairs of originals. Let $Z=\mathcal{T}^{-1}_T(Y)$. We can find $Z_1, Z_2$ such that $Z_1\cup Z_2=Z, \, Z_1 \cap Z_2=\emptyset$ and $\mathcal{T}_T|_{X\backslash Z_2}$ is bijective. Now we use that $\overline{\mu_T}$ is an invariant measure.
\[
\overline{\mu_T}(X)=\overline{\mu_T}(Z_2\cup (X\backslash Z_2 ))=\overline{\mu_T}(\overline{\mathcal{T}_T}( Z_2 )+\overline{\mu_T}(\overline{\mathcal{T}_T}(X\backslash  Z_2) )=\overline{\mu_T}(Y)+\overline{\mu_T}(X)
\]
which gives us that $\overline{\mu_T}(Y)=0$. Now we show that $\bigvee_{i=0}^\infty \mathcal{T}_T(\Omega \times [0,\infty))\cap X $ contains the fundamental intervals. A fundamental interval in $X$ has the form $\Delta(a_1,\ldots a_n)\times \Delta (b_1,\ldots,b_m) \cap X$ and
\begin{eqnarray*}
\mathcal{T}^m_T(\Delta(b_m,\ldots,b_1, a_1, \ldots, a_n) \times [0,\infty)) \cap X&&\\
= \Delta(a_1,\ldots, a_n) \times \Delta(b_1, \ldots, b_m) \cap X&&
\end{eqnarray*}
so $\bigvee_{i=0}^\infty \mathcal{T}_T(\Omega \times [0,\infty) )\cap X= \mathcal{B}(\Omega\times [0,\infty))\cap X$.
\end{proof}
Now we will show that $\cap_{i=0}^\infty \mathcal{T}_T^i(\Omega \times [0,\infty) )$ gives the right domain  in the case of two intervals (even in the non simple case). Let $l_1=\lfloor \frac{N_1}{a_1+1}\rfloor-a_2$ be the lowest digit of a branch on $[a_1,a_1+1)$ and $h_1=\lfloor \frac{N_1}{a_1}\rfloor-a_2-1$ the highest.  Define $l_2, h_2$ likewise. We first show that 

\begin{eqnarray*}
\mathcal{T}_T(\Omega \times [0,\infty)) &=&  [a_1,a_1+1) \times \left(0, \frac{N_2}{l_2}\right] \cup [a_2,a_2+1) \times \left(0, \frac{N_1}{l_1}\right] \\
\mathcal{T}_T^2(\Omega \times [0,\infty) )&=& [a_1,a_1+1) \times \left[\frac{N_2}{h_2+\frac{N_1}{l_1}}, \frac{N_2}{l_2}\right) \cup [a_2,a_2+1) \times \left[\frac{N_1}{h_1+\frac{N_2}{l_2}}, \frac{N_1}{l_1}\right) \\
\mathcal{T}_T^3(\Omega \times [0,\infty)) &=& [a_1,a_1+1) \times \left([N_2/h_2,N_1/h_1], [N_2/l_2,N_1/h_1,N_2/l_2] \right] \\
&& \cup [a_2,a_2+1) \times ([N_1/h_1,N_2/l_2], [N_1/l_1,N_2/h_2,N_1/l_1]]  \\
\vdots \hspace{1em} &&\\
\mathcal{T}_T^n(\Omega \times [0,\infty)) &=& [a_1,a_1+1) \times  \\
&& (\underbrace{[N_2/h_2,N_1/h_1,N_2/h_2,N_1/h_1,\ldots ]}_{n-1 \text{ terms if } n \text{ is odd, } n \text{ terms otherwise}}, \underbrace{[N_2/l_2,N_1/h_1,N_2/l_2,N_1/h_1,\ldots]}_{n \text{ terms if } n \text{ is odd, } n-1 \text{ terms otherwise}} ] \\
&& \cup [a_2,a_2+1) \times \\
&& (\underbrace{[N_1/h_1,N_2/l_2,N_1/h_1,N_2/l_2,\ldots]}_{n-1 \text{ terms if } n \text{ is odd, } n \text{ terms otherwise}}, \underbrace{[N_1/l_1,N_2/h_2,N_1/l_1,N_2/h_2,\ldots]}_{n \text{ terms if } n \text{ is odd, } n-1 \text{ terms otherwise}}]  .\\
\end{eqnarray*}
To show that this is indeed the case we need that on the intervals on the  $y$-coordinate of these rectangles have a length larger than $1$. This ensures that no gaps appear in between the cylinders (since in that case we have $\frac{N_i}{d +c_2}<\frac{N_i}{d+1 +c_1}$ for the interval $[c_1,c_2]$ and thus we find overlapping images for fundamental intervals). Now note that $[N_2/h_2,N_1/h_1,N_2/h_2,N_1/h_1,\ldots ]$ is a left convergent of $a_2$ when it has an even amount of terms. and $[N_2/l_2,N_1/h_1,N_2/l_2,N_1/h_1,\ldots]$ is a right convergent of $a_2+1$ when it has an odd amount of terms. Similarly $[N_1/h_1,N_2/l_2,N_1/h_1,N_2/l_2,\ldots]$ is a left convergent for $a_1$ when it has an odd amount of terms and $[N_1/l_1,N_2/h_2,N_1/l_1,N_2/h_2,\ldots]$ is a right convergent of $a_1+1$ when it has an even amount of terms. Therefore, if we write $ \mathcal{T}_T^n(\Omega \times [0,\infty))= [a_1,a_1+1) \times [b_n,b_n^*] \cup [a_2,a_2+1) \times [c_n, c_n^*]$ we have that $b_n^*-b_n>1$ and $ c_n^*-c_n>1$. Using the same observation we conclude that $X=\cap_{i=0}^\infty \mathcal{T}_T^i(\Omega \times [0,\infty) ) = a_1,a_1+1)\times [a_2,a_2+1]\cup [a_2,a_2+1)\times [a_1,a_1+1]$.\newline

\vspace{1em}

We will now numerically support our conjecture. Our algorithm works as follows: we determine the rectangles of $X_n:=\cap_{i=0}^n \mathcal{T}_T^i(\Omega \times [0,\infty) )$. On these rectangles we use the density $\frac{N}{(N+xy)^2}$, project this to the first coordinate and normalise it to find an approximation for the invariant density. We would like to remark that one can also find an approximation for the natural extension by taking a set of points in $\Omega\times \{0 \} $ and iterate these points. For various families of continued fractions this will give you the domain of the natural extension in the limit when taking the closure, see~\cite{KSS} for Nakada's $\alpha$-continued fractions and~\cite{EINN} for the Hurwitz complex continued fractions. The advantage of our method above is that it gives a smooth approximation of the invariant density in every step. Furthermore, our method finds a good approximation much faster.
First, we test our method in the case of two intervals. Let $a_1= 1$, $a_2= 2$ and $\textbf{N}=(12,12)$. Table~\ref{tab:wtheoretic} shows the difference between the densities found by the algorithm and the theoretical density. For comparison the table also includes a simulation of the density that is found by following orbits of typical points of the system (a more classical approach). See Figure~\ref{fig:oldversusnew} for a graphical comparison. Note that the graphs are so close to each other that it is difficult to see that they are different.

\begin{figure}[h!]
  \centering
{\includegraphics[width=0.45\textwidth]{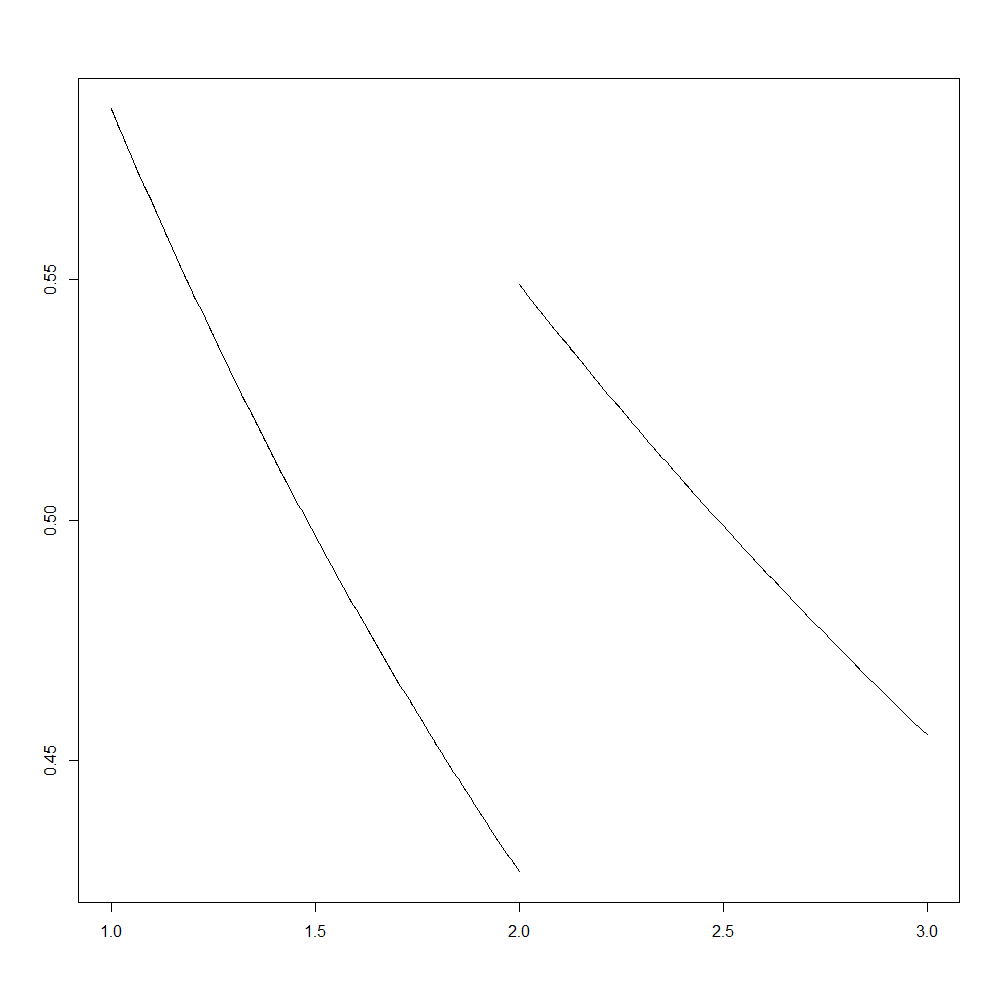}}
  \hfill
{\includegraphics[width=0.45\textwidth]{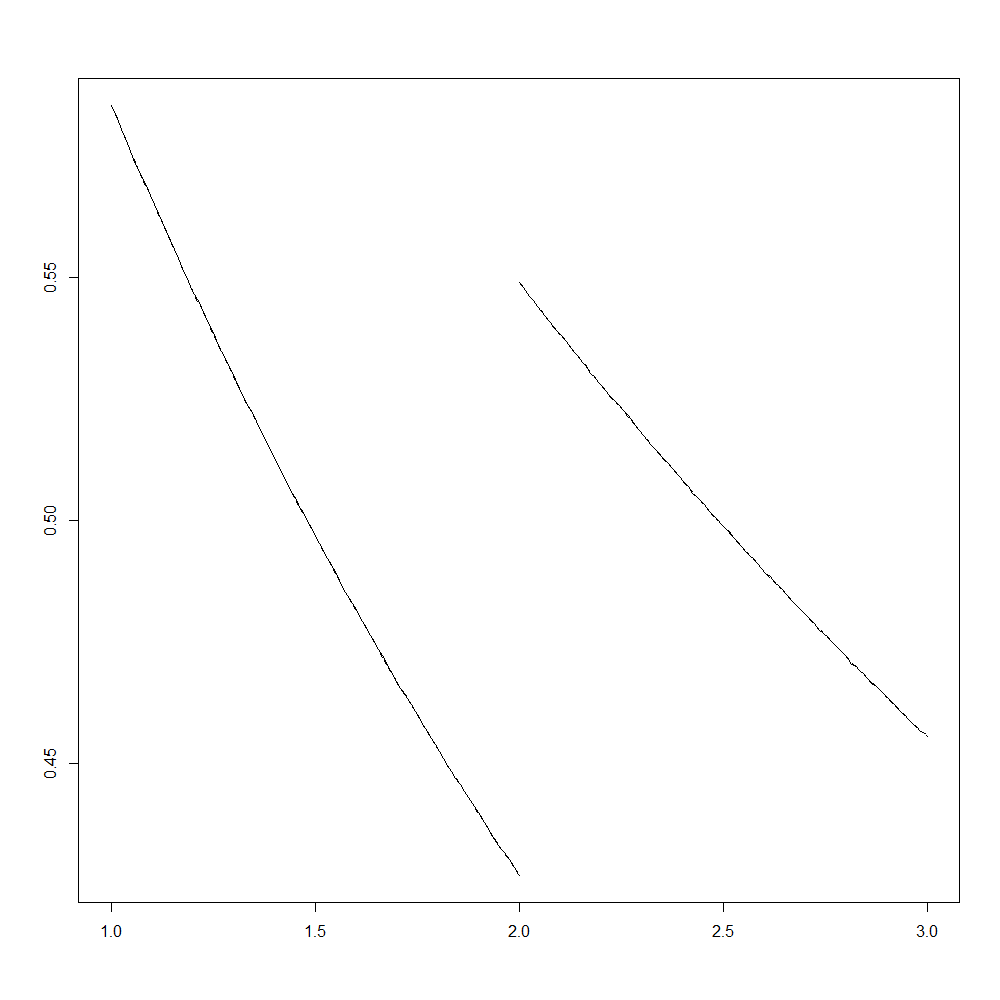}}
  \caption{On the left is the new method plotted against the theoretical density (dashed line). On the right is the classical method plotted against the theoretical  density (dashed line).}\label{fig:oldversusnew}
\end{figure}

\begin{table}[h!]
\begin{center}
\begin{tabular}{ c | c| c| c| c| c| c| }
 number of iterations & 0 & 1 & 2 & 3 & 7 & simulation \\ 
 \hline 
 difference  in $|| \, .  \, ||_1$ norm   & 0.52859 & 0.01763 & 0.00625 & 0.00238 & 4.98144e-05 & 0.00023 \\ 
\end{tabular}
\vspace{0.5em}
\caption{Number of iterations and the difference in $|| \, .  \, ||_1$ norm with  the theoretical density. For the normal simulation the computer iterated points for 1 hour (the new method is calculated within a second).}
\label{tab:wtheoretic}
\end{center}
\end{table}

In case we have more than two intervals we are unable to compare the approximations with the true invariant density. We can compare the approximations with a simulation of the density that is found by following orbits of typical points. For the previous example, comparing the seventh approximation with a normal simulation gives $0.0002379945$. 

\vspace{1em}
Let us now turn to a new example. Take $a_1=1$, $a_2=3$, $a_3=2$ and $\textbf{N}=(12,12,12)$, see Figure~\ref{fig:Texamplesim}. The result of the comparison of the method with a normal simulation can be found in Table~\ref{tab:wotheoretic}. 
To further justify our conjecture we can look at $r_n=\frac{\hat{\mu}_T(X_n)-\hat{\mu}_T(X_{n+1})}{\hat{\mu}_T(X_n)}$ where $\hat{\mu}(X_n)$ is the non normalised version of $\overline{\mu_T}$. This will give an indication of the fraction of mass that one 'throws away in the iteration'. If we want to have that $X=\cap_{i=0}^\infty \mathcal{T}_T^i(\Omega \times [0,\infty) )$ has positive measure we need that $r_n\rightarrow 0$ as $n\rightarrow \infty$. In Table~\ref{tab:thrownaway} we find the result and see that this indeed seems to be the case.

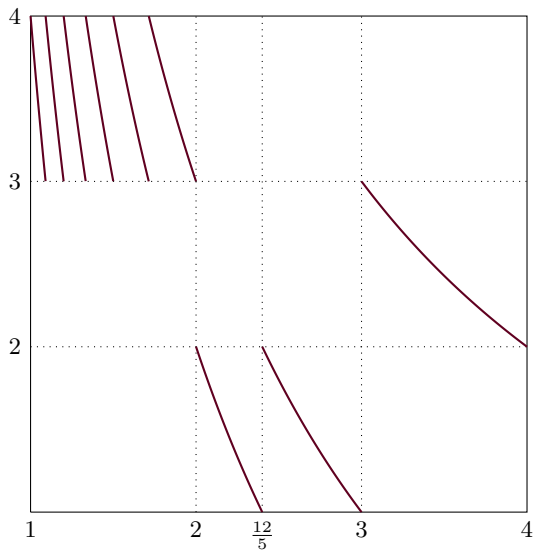
\begin{figure}[h]
\centering
\begin{tikzpicture}[scale=2.2]

\draw(1,1)node[below]{\small $1$}--(4,1)node[below]{\small $4$}--(4,4)--(1,4)node[left]{\small $4$}--(1,1);
 \draw[dotted] (2.4,1)node[below]{\small $\frac{12}{5}$}--(2.4,4);

\draw[dotted] (2,1)node[below]{\small $2$}--(2,4);
\draw[dotted] (1,2)node[left]{\small $2$}--(4,2);
\draw[dotted] (3,1)node[below]{\small $3$}--(3,4);
\draw[dotted] (1,3)node[left]{\small $3$}--(4,3);

\draw[thick, purple!50!black, smooth, samples =20, domain=3:4] plot(\x,{(12 /\x -1)});

\draw[thick, purple!50!black, smooth, samples =20, domain=2.4:3] plot(\x,{(12 /\x -3)});
\draw[thick, purple!50!black, smooth, samples =20, domain=2:2.4] plot(\x,{(12 /\x -4)});

\draw[thick, purple!50!black, smooth, samples =20, domain=12/7:2] plot(\x,{(12 /\x -3)});
\draw[thick, purple!50!black, smooth, samples =20, domain=12/8:12/7] plot(\x,{(12 /\x -4)});
\draw[thick, purple!50!black, smooth, samples =20, domain=12/9:12/8] plot(\x,{(12 /\x -5)});
\draw[thick, purple!50!black, smooth, samples =20, domain=12/10:12/9] plot(\x,{(12 /\x -6)});
\draw[thick, purple!50!black, smooth, samples =20, domain=12/11:12/10] plot(\x,{(12 /\x -7)});
\draw[thick, purple!50!black, smooth, samples =20, domain=1:12/11] plot(\x,{(12 /\x -8)});

\end{tikzpicture}
\caption{The map $T$ for $a_1=1$, $a_2=3$, $a_3=2$ and $\textbf{N}=(12,12,12)$.}
\label{fig:Texamplesim}
\end{figure}

\begin{table}[h!]
\begin{center}
\begin{tabular}{ c | c| c| c| c| c|  }
 number of iterations & 0 & 1 & 2 & 3 & 7 \\ 
 \hline 
 difference  in $|| \, .  \, ||_1$ norm   &  0.0495665  &  0.01142  &  0.01029  &  0.00294 &  0.00037 \\ 
\end{tabular}
\vspace{0.5em}
\caption{Number of iterations and the difference in $|| \, .  \, ||_1$ norm with  a normal density. For the normal simulation the computer iterated points for 1 hour (the new method is calculated within a second).}
\label{tab:wotheoretic}
\end{center}
\end{table}

\begin{table}[h!]
\begin{center}
\begin{tabular}{ c | c| c| c| c| c|  }
 number of iterations n & 0 & 1 & 2 & 3 & 7 \\ 
 \hline 
 $r_n$ &  0.52052  &  0.38379  & 0.38686  &  0.21580 &  0.08922 \\ 
\end{tabular}
\vspace{0.5em}
\caption{$r_n=\frac{\hat{\mu}_T(X_n)-\hat{\mu}_T(X_{n+1})}{\hat{\mu}_T(X_n)}$ gives an indication how much mass is thrown away at each iteration.}
\label{tab:thrownaway}
\end{center}
\end{table}
Now we can improve our approximation in the following way. We can calculate rectangles in which we know the domain must be a subset of and use this as an initial domain instead of $\Omega \times [0,\infty)$. We first determine $L,H$ such that the natural extension domain for $x\in[a_1,a_1+1]$ is a subset of $[a_1,a_1+1]\times [L,H]$. Afterwards we take  image of this rectangle to find the other rectangles. For $L$ we find the periodic point \newline $[\overline{N/h_2,N/l_3,N/h_1,N/l_2,N/h_3,N/l_1}]=[\overline{12/4,12/1,12/8,12/3,12/1,12/3}]=\frac{75}{59}$ and for $H$ we find \newline
$[\overline{N/l_2,N/h_3,N/l_1,N/h_2,N/l_3,N/h_1}]=[\overline{12/3,12/1,12/3,12/4,12/1,12/8}]=\frac{80}{41}$.
This gives us the first block $[a_1,a_1+1]\times [\frac{75}{59},\frac{80}{41}]$, the second block we find is 
\[
\mathcal{T}_T([a_1,a_1+1]\times [\frac{75}{59},\frac{80}{41}])=[a_3,a_3+1]\times [\frac{41}{34},\frac{59}{21}]
\]
and the third block is given by 
\[
\mathcal{T}_T^2([a_1,a_1+1]\times [\frac{75}{59},\frac{80}{41}])=[a_3,a_3+1]\times [\frac{63}{20},\frac{136}{25}]
\]

When we compare this with a simulation of the density which is found by following orbits of typical points of the system we get the results in Table~\ref{tab:withbetterinitial}. See Figure~\ref{fig:rect} for the domains $X_0,X_1,X_7$ and Figure~\ref{fig:densunknown} for the approximation of the invariant density after 7 iterations. Note that the differences in Table~\ref{tab:withbetterinitial} are not drastically changing. For the two methods to be closer to each other it is likely that one needs a more precise approximation for the classical simulation.

\begin{table}[h!]
\begin{center}
\begin{tabular}{ c | c| c| c| c| }
 number of iterations & 0 & 1 & 2 & 7 \\ 
 \hline 
 difference  in $|| \, .  \, ||_1$ norm   &  4.794264e-04  &  4.001946e-04  &  3.304742e-04 & 3.263492e-04 \\ 
\end{tabular}
\vspace{0.5em}
\caption{Number of iterations and the difference in $|| \, .  \, ||_1$ norm with  a normal density with using a better initial state. For the normal simulation the computer iterated points for 1 hour (the new method is calculated within a second).}
\label{tab:withbetterinitial}
\end{center}
\end{table}

\begin{figure}[h!]
  \centering
{\includegraphics[width=0.30\textwidth]{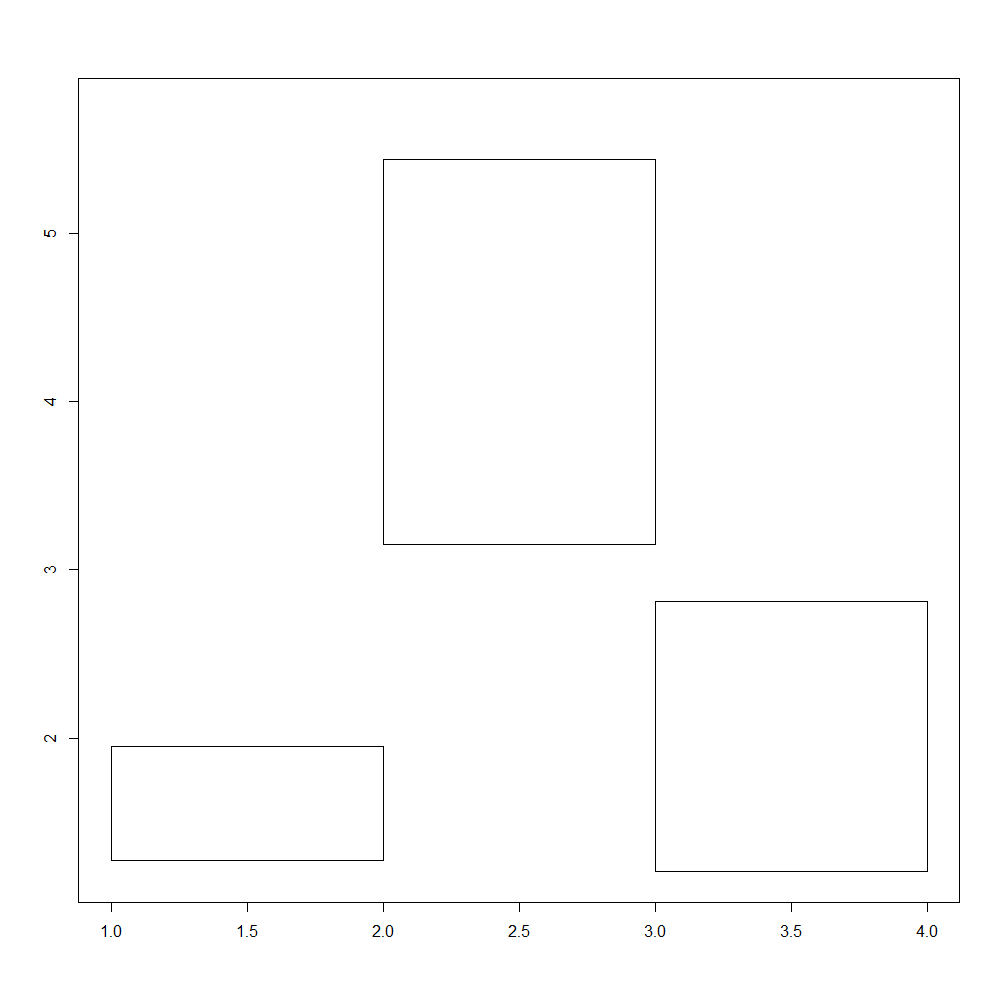}}
  \hfill
{\includegraphics[width=0.30\textwidth]{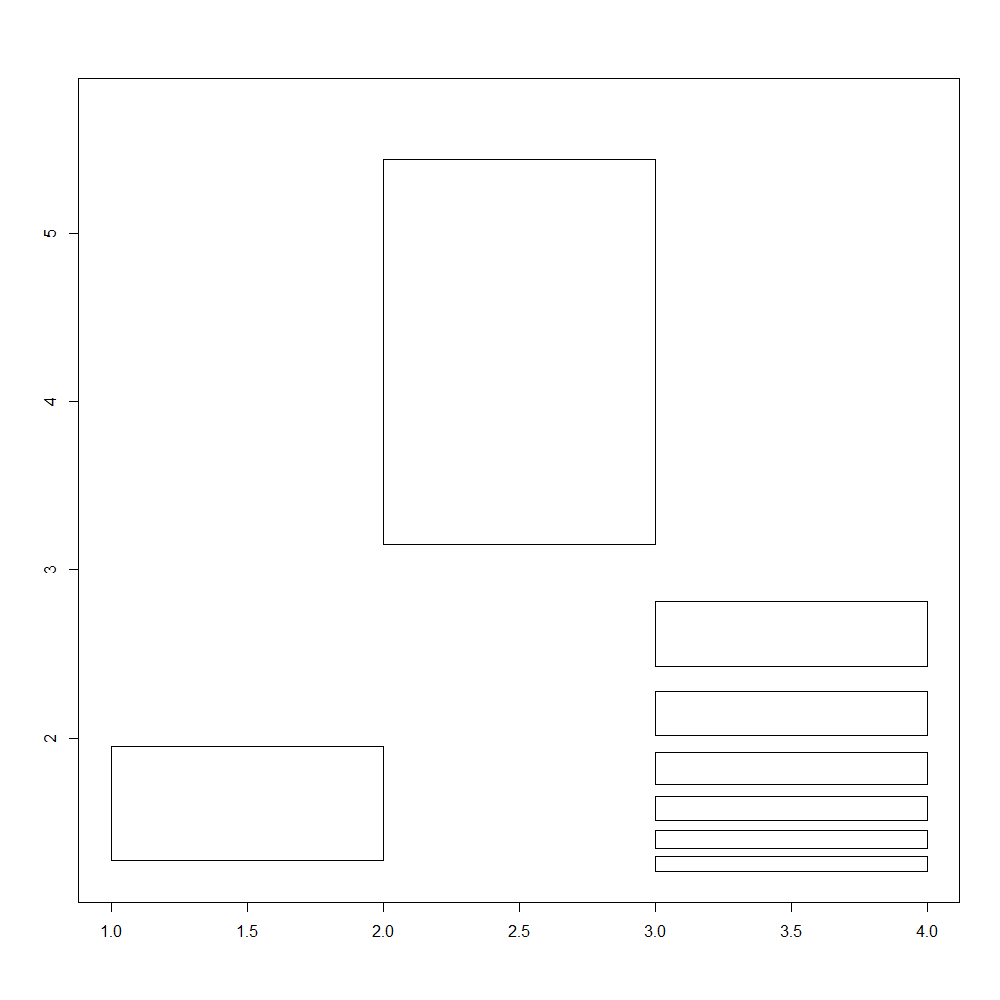}}
  \hfill
{\includegraphics[width=0.30\textwidth]{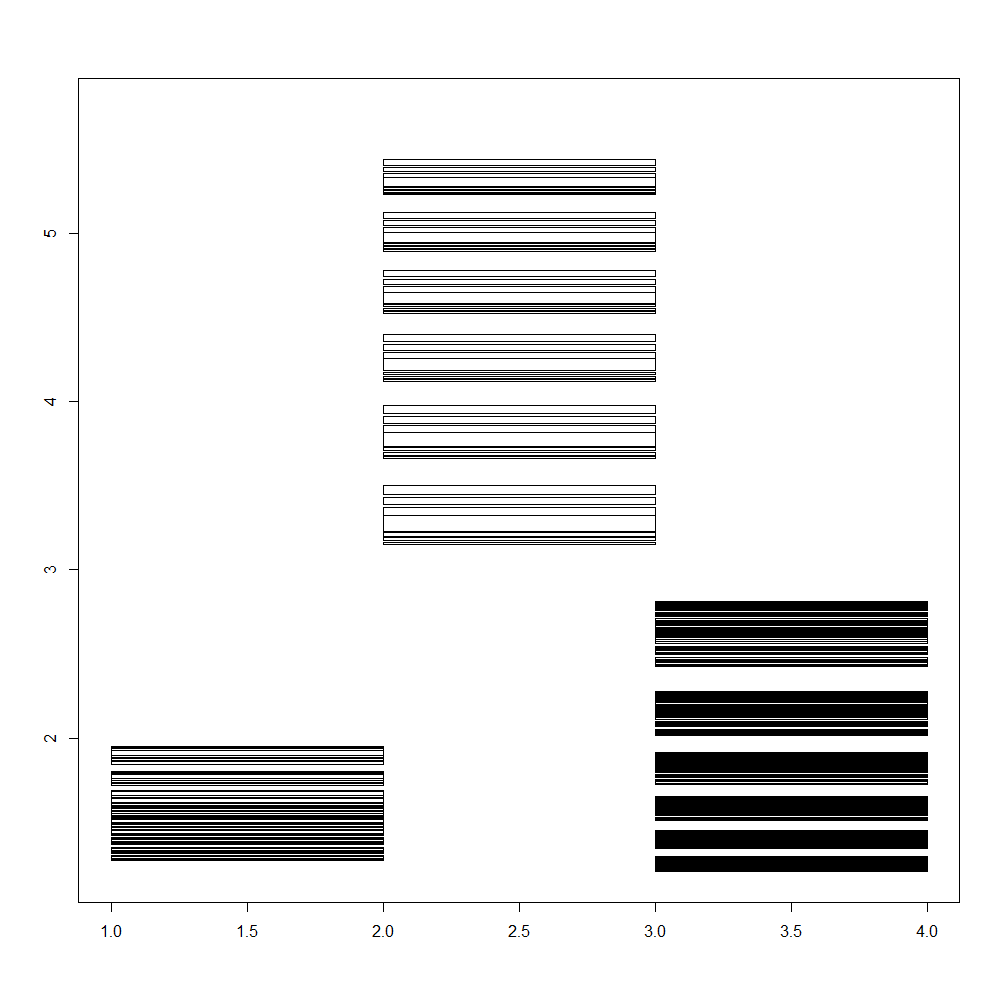}}
  \caption{From left to right, the domains $X_0,X_1$ and $X_7$.}\label{fig:rect}
\end{figure}

\begin{figure}[h!]
  \centering
{\includegraphics[width=0.5\textwidth]{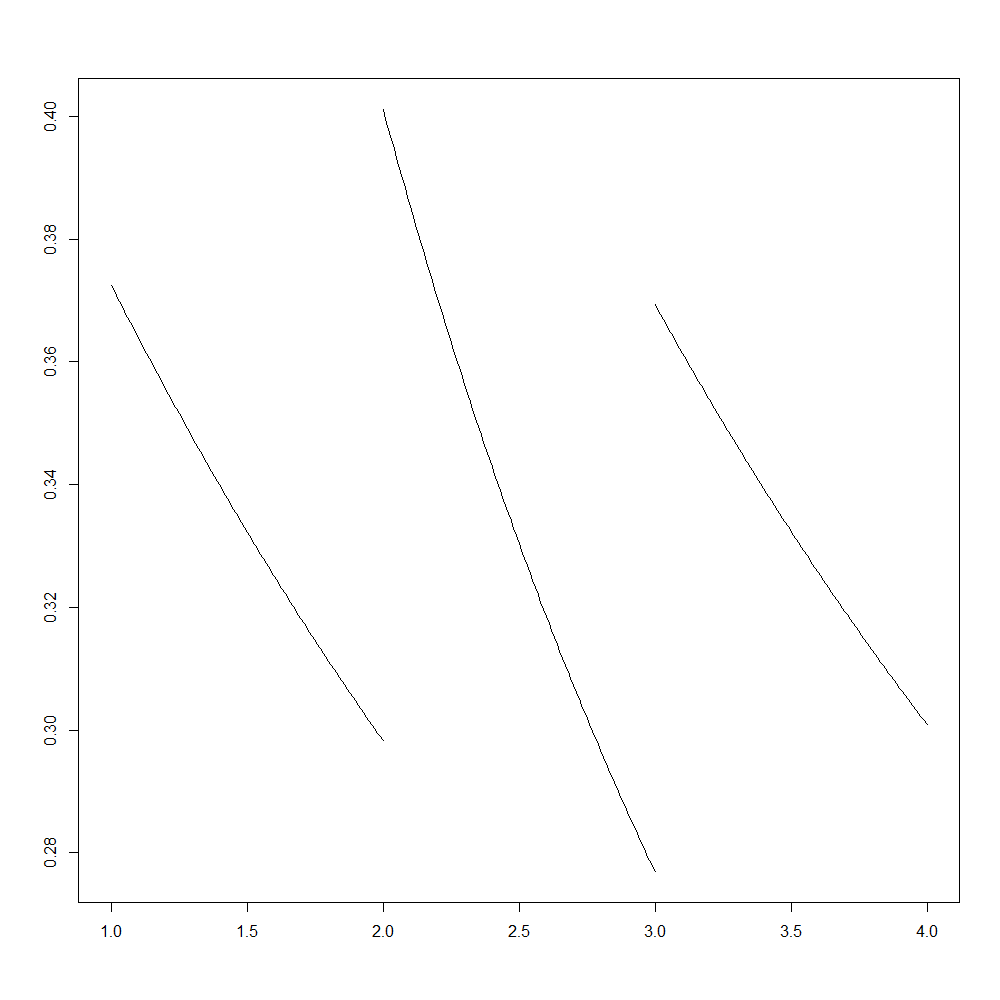}}
  \caption{The simulated density using $X_7$.}\label{fig:densunknown}
\end{figure}

Another advantage of our method is that, if our conjecture holds, in each step of the algorithm you obtain rigorous bounds for $\theta_n(x)$ for all $x\in \Omega$ similar to (\ref{ineq:thetha1}) and (\ref{ineq:thetha2}). Furthermore, when we improve our initial state and find an $L$ and $H$ as we did in our last example these bounds will be sharp.

\subsection{For our examples}
Let us now return to the examples of Section~\ref{sec:examples}. For Example 1 we know the system is ergodic, but we do not know the  domain of the natural extension and or the invariant density. In Figure~\ref{fig:examplesnatex} on the left we see a simulation of the domain. For Example 2 we know the domain of the natural extension namely $X= [1,2]\times [2,3]\cup [2,3]\times [1,2]$ but we do not know the density. For example 3 we do know the density but the domain of the natural extension does not seem to have a simple structure (i.e. is not a union of finitely many rectangles), see Figure~\ref{fig:examplesnatex} on the right. 

\begin{figure}[h!]
  \centering
{\includegraphics[width=0.45\textwidth]{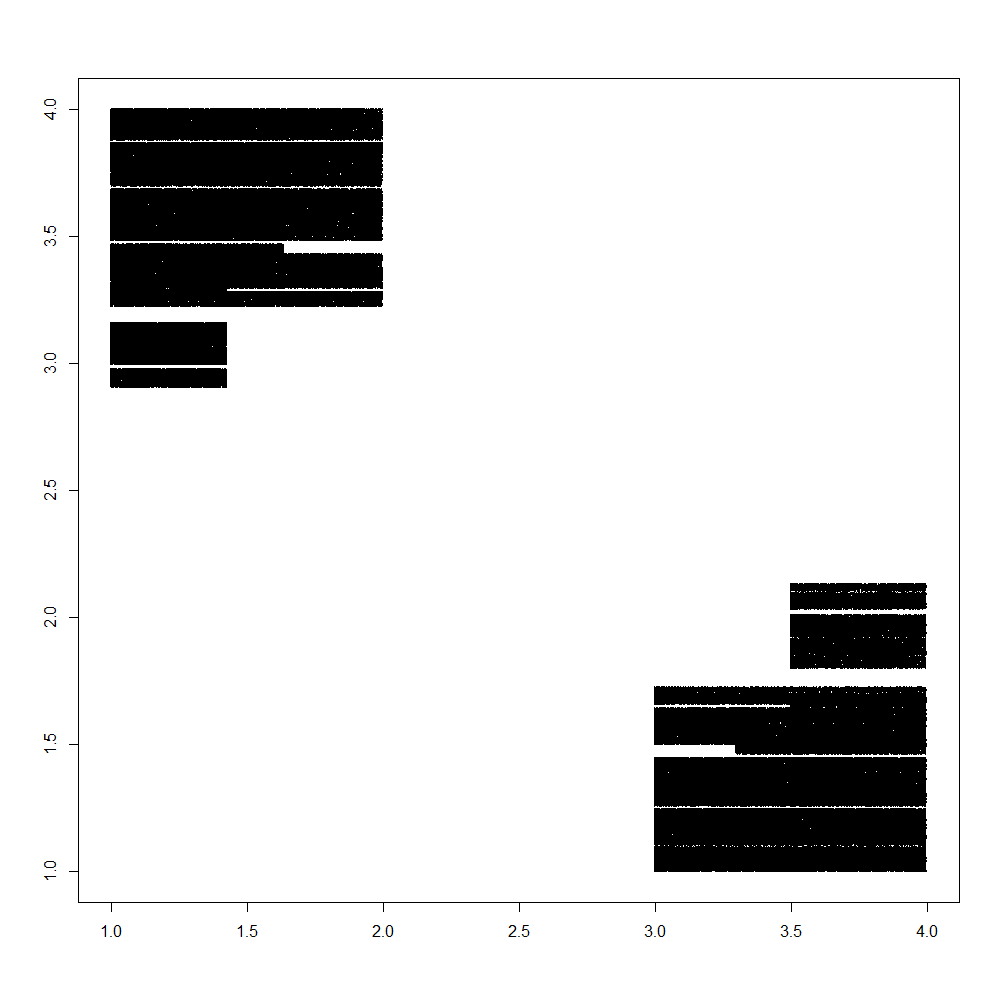}}
  \hfill
{\includegraphics[width=0.45\textwidth]{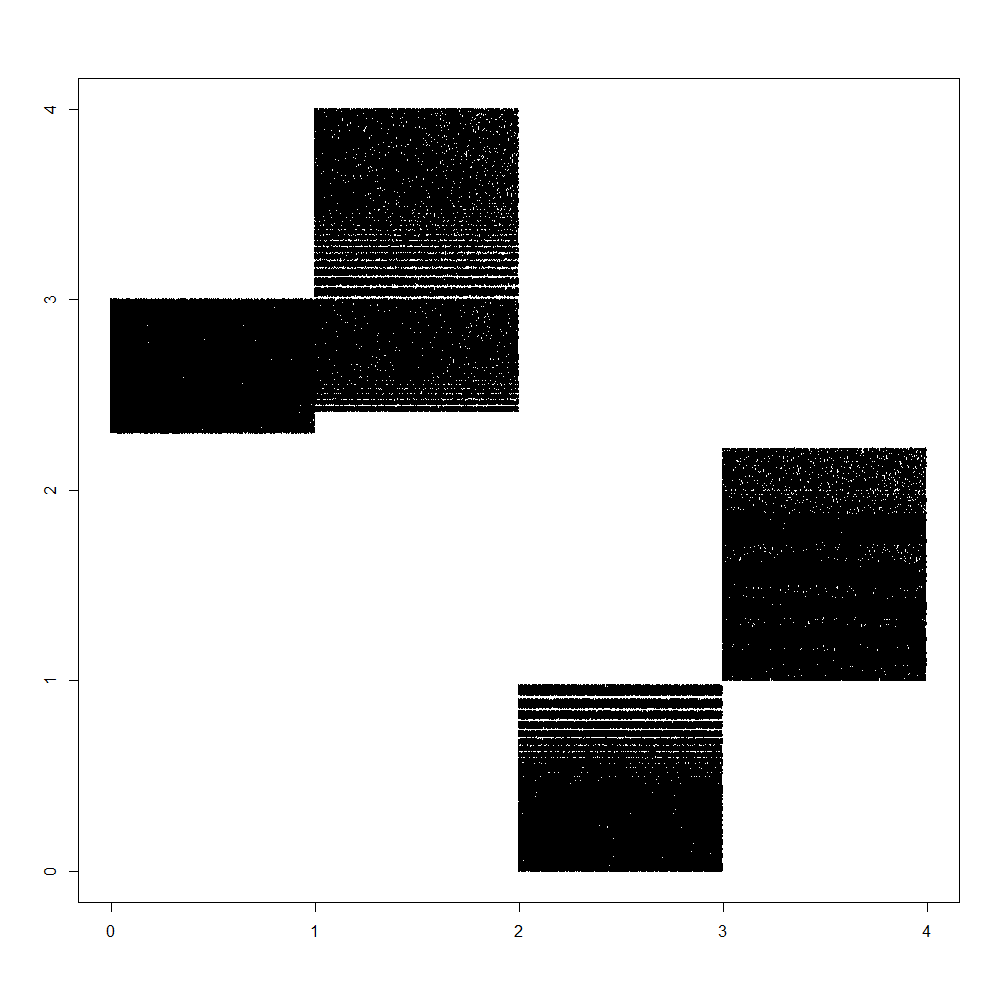}}
  \caption{Simulations of the domain of Example 1 on the left and Example 3 on the right.}\label{fig:examplesnatex}
\end{figure}

\section{Final Remarks and Open Questions}
First we want to mention that the proof for convergence and the proof of ergodicity did not depend on taking intervals with integer endpoints. When taking $N\in\mathbb{R}_{>1}$ the same proofs also hold.  \newline
One goal that we have not achieved is in finding the absolutely continuous invariant measure for the non simple case and for the case we have more than two intervals. \newline
A direction we did not investigate is in determining the symbolic spaces we can get with the systems we studied. What follows are some remarks on the matter. First we will look at integer vectors $\textbf{N}$ for which we can find a continued fraction algorithm. Note that for all $(N_1,N_2)$ we can find an allowable system by choosing $a_1=0,a_2=1$ and putting $N_1$ on $[a_1,a_1+1)$ and $N_2$ on $[a_2,a_2+1)$. This is allowed since $N_1>1$ so $\lfloor \frac{N_1}{1}\rfloor-1>0$ and since $N_2>1$ we have $\lfloor \frac{N_2}{2}\rfloor>0$. Whenever $N_1$ or $N_2$ is even we can choose to put the even one on $[1,2)$ and find a desirable case. In general we can ask: for a fixed $N_1$ for which values of $N_2$ can we find a continued fraction algorithm that is desirable? When putting $N_1$ on $[0,1)$ we will get the most choices of different intervals. Choices for $N_2$ to make a desirable continued fraction algorithm are given by multitudes of the numbers $\{2,6,12,\ldots, n(n+1),\ldots, (N_1-1)N_1\}$. Note that if there is no $n\in\mathbb{N}$ such that both $n$ and $n+1$ are divisors of $N_1$ then we have to put it on $[0,1)$ to make the system desirable. If there is such an $n$ we can add for all such $n$ the choice for $N_2>n$ and put this on $[0,1)$. It is easy to see that if we want to have a simple system ($N_1=N_2$) then it is necessary and sufficient that there is an $n\in\mathbb{N}$ such that $n$ and $n+1$ are divisors of $N_1$.

\bibliographystyle{alpha}
\bibliography{alternatingNexpansions}

\end{document}